\documentclass[reqno, 12pt, reqno]{amsart}
\usepackage{amsfonts, amsthm, amsmath, amssymb}
\usepackage{hyperref}
\usepackage[all]{xy} 
\usepackage[margin=1.5in]{geometry}
\usepackage{multirow}
\usepackage{helvet}
\RequirePackage{mathrsfs} \let\mathcal\mathscr
\usepackage{verbatim} 

\makeatletter
\def\@tocline#1#2#3#4#5#6#7{\relax
  \ifnum #1>\c@tocdepth 
  \else
    \par \addpenalty\@secpenalty\addvspace{#2}%
    \begingroup \hyphenpenalty\@M
    \@ifempty{#4}{%
      \@tempdima\csname r@tocindent\number#1\endcsname\relax
    }{%
      \@tempdima#4\relax
    }%
    \parindent\z@ \leftskip#3\relax \advance\leftskip\@tempdima\relax
    \rightskip\@pnumwidth plus4em \parfillskip-\@pnumwidth
    #5\leavevmode\hskip-\@tempdima
      \ifcase #1
       \or\or \hskip 1em \or \hskip 2em \else \hskip 3em \fi%
      #6\nobreak\relax
    \dotfill\hbox to\@pnumwidth{\@tocpagenum{#7}}\par
    \nobreak
    \endgroup
  \fi}
\makeatother

\usepackage{fancyref} 

\usepackage[utf8]{inputenc}
\numberwithin{equation}{section}

\newtheorem{theorem}{Theorem}[section]
\newtheorem{lemma}[theorem]{Lemma}

\newtheorem{corollary}[theorem]{Corollary}

\theoremstyle{definition}

\usepackage[english]{babel} 
\usepackage{enumerate} 
\usepackage{relsize} 
\allowdisplaybreaks 
\DeclareMathOperator{\Mor}{Mor}

\DeclareMathOperator{\ord}{ord} 
\DeclareMathOperator{\Tr}{Tr} 
 
\DeclareMathOperator{\meas}{meas} 
\DeclareMathAlphabet{\xcal}{OMS}{cmsy}{m}{n} 
\DeclareMathOperator{\ch}{char} 

\newcommand*\diff{\mathop{}\!\mathrm{d}} 
 
\newcommand{\Mod}[1]{\ \text{mod}\ #1} 

\newcommand{\Sum}{\mathlarger{\sum}} 
\newcommand{\Suma}{\mathlarger{\sideset{}{^*}\sum}_{|a|<|r|}}

\renewcommand{\v}[1]{{\bold{#1}}} 
\newcommand{\Fp}{\mathbb{F}_{p}}
\newcommand{\Fq}{\mathbb{F}_{q}}

\newcommand{\whC}{{\widehat C}}

\newcommand{\whN}{{\widehat N}}

\newcommand{\whQ}{{\widehat Q}}
\newcommand{\whR}{{\widehat R}}

\newcommand{\whY}{{\widehat Y}}

\newcommand{\whTheta}{{\widehat \Theta}}

\newcommand{\mbC}{{\mathbb C}}

\newcommand{\mbN}{{\mathbb N}}

\newcommand{\mbP}{{\mathbb P}}

\newcommand{\mbR}{{\mathbb R}}

\newcommand{\mbT}{{\mathbb T}}

\newcommand{\mbZ}{{\mathbb Z}}

\renewcommand{\leq}{\leqslant}
\renewcommand{\geq}{\geqslant}

\begin{document}

\title{Rational curves on cubic hypersurfaces over finite fields}

\author{Adelina\ M\^{a}nz\u{a}\c{t}eanu}
\address{School of Mathematics\\
University of Bristol\\ Bristol\\ BS8 1TW\\ UK}
\email{am15112@bristol.ac.uk}

\begin{abstract}
Given a smooth cubic hypersurface $X$ over a finite field of characteristic greater than 3 and two generic points on $X$, we use a function field analogue of the Hardy--Littlewood circle method to obtain an asymptotic formula for the number of degree $d$ rational curves on $X$ passing through those two points. We use this to deduce the dimension and irreducibility of the moduli space parametrising such curves, for large enough $d$.
\end{abstract}

\date{\today}


\maketitle

\tableofcontents

\thispagestyle{empty}
\section{Introduction}  
Let $k = \Fq$ be a finite field and let $F \in k[x_1, \ldots, x_n]$ denote a non-singular homogeneous polynomial of degree $3$. Moreover, let $X \subset \mathbb{P}_k^{n-1}$ be the smooth cubic hypersurface defined over $k$ by $F=0$. Let $C$ be a smooth projective curve over $k$. Then $k(C)$ has transcendence degree 1 over a $C_1$-field and, by the Lang--Tsen theorem \cite[Theorem 3.6]{Gre69}, the set $X(k(C))$ of $k(C)$-rational points on $X$ is non-empty for $n \geq 10$. This still holds for $X$ singular. In this paper we are interested in the case $C = \mathbb{P}^1$, writing $K = k(t)$ denote the function field of $C$ over $k$. 
A \emph{degree $d$ rational curve on $X$} is a non-constant morphism $f:\mbP_{k}^1 \to X$ given by
\begin{equation}\label{rational curve}
f=\left(f_1(u,v), \ldots, f_n(u,v)\right),
\end{equation}
where $f_i\in \bar{k}[u, v]$ are homogeneous polynomials of degree $d \geq 1$, with no non-constant common factor in $\bar{k}[u, v]$, such that
$$
F\left(f_1(u,v),\ldots, f_n(u,v)\right) \equiv 0.
$$
Such a curve is said to be \emph{$m$-pointed} if it is equipped with a choice of $m$ distinct points $P_1, \ldots, P_m \in X(k)$ called the \emph{marks} through which the curve passes. Up to isomorphism, these curves are parametrised by the moduli space $\xcal{M}_{0,m}(\mbP_k^1,X,d)$. The compactification of this space,  $\overline{\xcal{M}}_{0,m}(\mbP_k^1,X,d)$, is the Kontsevich moduli space of stable maps. 

Suppose from now on that $\# k = q$ and $\ch(k) > 3$. In \cite[Example 7.6]{Kol08}, Koll\'{a}r proves that there exists a constant $c_n$ depending only on $n$ such that for any $q > c_n$ and any point $x \in X(k)$, there exists a $k$-rational curve of degree at most 216 on $X$ passing through $x$. In our investigation we focus on the case $m = 2$ of 2-pointed rational curves on $X$.

Associate to $F$ the \emph{Hessian matrix}
$$
\v{H}(\v{x}) = \left(\frac{\partial^2 F}{\partial x_i \partial  x_j} \right)_{1\leq i, j \leq n}
$$
and the \emph{Hessian hypersurface} $H=0$, where $H(\v{x}) = \det \v{H}(\v{x}).$ Now let $\v{a}, \v{b} \in \Fq^n \setminus \{\v{0}\}$ be such that $F(\v{a})=F(\v{b})=0$ and $ {H}(\v{b}) \neq 0.$ Write $a = [\v{a}]$, $b = [\v{b}]$ for the corresponding points in $X(k)$. As is well-known (see \cite[Lemma 1]{Hoo88}, for example), the Hessian $H(\v{x})$ does not vanish identically on $X$, since $\ch(k)>3$. 
The main goal of this paper is to obtain an asymptotic formula for the number of rational curves of degree $d$ on $X$ passing through $a$ and $b$. Denote the space of such curves by $\Mor_{d, a, b} (\mbP_{k}^1,X)$. We can write the $f_i$ in (\ref{rational curve}) explicitly as 
$$
f_i(u,v) = \alpha_d^{\left(i\right)}u^d +\alpha_{d-1}^{\left(i\right)}u^{d-1}v + \ldots + \alpha_0^{\left(i\right)}v^d,
$$
where $\alpha_j^{\left(i\right)}\in k$ for $0 \leq j \leq d$ and $1 \leq i \leq n$. Then, we capture the condition that the rational curve $f$ passes through the points $\v{a}$ and $\v{b}$ by selecting
\begin{equation}\label{eq: conditions from a,b}
\begin{aligned}
 \left(\alpha_0^{(1)}, \ldots, \alpha_0^{(n)}\right ) =\v{a},\\ 
 \left(\alpha_d^{(1)}, \ldots, \alpha_d^{(n)}\right ) =\v{b}. 
\end{aligned}
\end{equation}

There exists a correspondence between the rational curves on $X$ of bounded degree and the $K$-points on $X$ of bounded height. Define $N_{a,b}(d)$ to be the number of polynomials $f_1, \ldots, f_n \in \Fq[t]$ of degree at most $d$ whose constant coefficients are given by $\v{a}$ and whose leading coefficients are given by $\v{b}$, such that $F(f_1, \ldots, , f_n) = 0$. Thus, $N_{a,b}(d)$ counts the $\Fq$-points $(f_1, \ldots, f_n)$ on the affine cone of $\Mor_{d,a,b}(\mbP^1_k,X)$, where the condition that $f_1, \ldots, f_n$ have no common factor is dropped. Using a version of the Hardy--Littlewood circle method for the function field $K$ developed by Lee \cite{Lee11,Lee13}, and further by Browning--Vishe \cite{BVBlue}, we shall obtain the following result.
\begin{theorem}\label{MainResult} 
Fix $k=\Fq$ with $\mathrm{char}(k)> 3$. Fix a smooth cubic hypersurface $X \subset \mbP^{n-1}_k$, where $n \geq 10$. Let $a,b \in X(k)$, not both on the Hessian. Then, we have
$$
N_{a,b}(d) = q^{(d-1)n-(3d-1)} + O\left(q^{\frac{5(d+2)n}{6}-\frac{5d+16}{3}} + q^{\frac{(5d+8)n}{6}-\frac{3d}{2}-\frac{14}{3}} + q^{\frac{3(d+5)n}{4}-\frac{3(d+5)}{4}}  \right),
$$
where the implied constant in the estimate depends only on $d$ and $X$. 
\end{theorem}

The condition that one of the two fixed points is not on the Hessian comes from our analysis of certain oscillatory integrals (see Lemma \ref{Lemma7.3inBlue}). 

Although it would be possible to generalise Theorem \ref{MainResult} to handle rational curves passing through any generic finite set of points in $X(k)$, the main motivation for considering rational curves through two fixed points comes from the notion of rational connectedness. In \cite{Man72}, Manin defined $R$-equivalence on the set of rational points of a variety in order to study the parametrisation of rational points on cubic surfaces. We say that two points $a, b \in X(k)$ are \emph{directly R-equivalent} if there is a morphism $f  : \mathbb{P}^1 \to X $ (defined over $k$) with $f(0,1) = a$  and $f(1,0) = b$; the generated equivalence relation is called \emph{R-equivalence}. In \cite{SD81}, Swinnerton-Dyer proved that $R$-equivalence is trivial on smooth cubic surfaces over finite fields; that is, all $k$-points are $R$-equivalent. Next, the result was generalised for smooth cubic hypersurfaces $X \subset \mathbb{P}^{n-1}_k$, if $n \geq 6$ by Madore in \cite{Mad08}, and if $n \geq 4$ and $q \geq 11$ by Koll\'{a}r in \cite{Kol08}. Moreover, Madore's result holds for $X$ defined over any $C_1$ field. The study of $R$-equivalence is closely related to understanding the geometry of the moduli space of rational curves. In particular, it is interesting to study $R$-equivalence in the case of varieties with many rational curves. Such varieties are called \emph{rationally connected} and were first studied by Koll\'{a}r, Miyaoka and Mori in \cite{KMM}, and independently by Campana in \cite{Cam92}. Roughly speaking, $Y$ is rationally connected if for two general points of $Y$ there is a rational curve on $Y$ passing through them. Thus, rationally connected varieties are varieties for which $R$-equivalence becomes trivial when one extends the ground field to an arbitrary algebraically closed field. Note that in the case of fields of positive characteristic one should consider \emph{separably rationally connected} varieties. For precise definitions and a thorough introduction to the theory see Koll\'{a}r \cite{Kol96}, \cite{Kol01}, and Koll\'{a}r--Szab\'{o} \cite{KollarSzabo}. 
\begin{corollary}
Fix $k=\Fq$ with $\ch(k)> 3$. Fix a smooth cubic hypersurface $X \subset \mbP^{n-1}_k$, where $n \geq 10$. Then there exists a constant $c_X >0$ such that for any points $a, b \in X(k)$, not both on the Hessian, and any $d \geq \frac{19(n-1)}{n-9}$, if $q \geq c_X$, then there exists a $\Fq$-rational curve $C \subset X$ of degree $d$ that passes through $a$ and $b$.
\end{corollary}
This can also be seen as a corollary of Pirutka \cite[Proposition 4.3]{Pir11} which states that any two points $a,b \in X(k)$ can be joined by two lines on $X$ defined over $k$.  

Keeping track of the dependance on $q$ allows us to deduce further results regarding the geometry of the moduli space $\Mor_{d,a,b}(\mbP^1_{k},X)$, in the spirit of those obtained by Browning--Vishe \cite{BV2}. We can regard $f$ in (\ref{rational curve}) under the conditions given by (\ref{eq: conditions from a,b}) as a point in $\mbP_k^{n(d-1)-1}$. Then the space $\Mor_{d,a,b}(\mbP^1_k,X)$ is an open subvariety of $\mbP_k^{n(d-1)-1}$ cut out by $3d-1$ equations and so has expected naive dimension $\mu = (n-3)d -n$. 
\begin{corollary}\label{Cor:leq1}
Fix $k=\Fq$ of $\ch(k)> 3$. Fix a smooth cubic hypersurface $X \subset \mbP^{n-1}_k$, where $n \geq 10$. Pick any points $a, b \in X(k)$, not both on the Hessian. Then for $d \geq \frac{19(n-1)}{n-9}$ we have
$$
\lim_{q\to\infty}q^{-\hat{\mu}} N_{a,b}(d) \leq 1,
$$
where $\hat{\mu} = \mu + 1$.
\end{corollary}

A result similar to \cite[Theorem 2.1]{BV2} concerning $\Mor_{d,a,b}(\mbP^1_{k},X)$ follows from Corollary \ref{Cor:leq1}. Now, by \cite[Theorem II.1.2]{Kol96}, all irreducible components of $\Mor_{d,a,b}(\mbP^1_{k},X)$ have dimension at least $\mu$. Then, comparing this with the Lang--Weil estimate \cite{LangWeil}, we obtain that the space $\Mor_{d,a,b}(\mbP^1_{k},X)$ is irreducible and of expected dimension $\mu$. Following the same ``spreading out'' argument (see \cite[\S 10.4.11]{Groth} and \cite{Serre}) as in \cite[\S 2]{BV2}, the problem over $\mbC$ can be related to the problem over $\Fq$, and this leads to the following corollary. 
\begin{corollary}
Fix a smooth cubic hypersurface $X \subset \mbP^{n-1}$ defined over $\mbC$, where $n \geq 10$. Pick any points two points in $X(\mbC)$, not both on the Hessian. Then for each $d \geq \frac{19(n-1)}{n-9}$, the space $\xcal{M}_{0,2}(\mbP_{\mbC}^1, X, d)$ is irreducible and of expected dimension $\bar{\mu} = \mu - 3$.
\end{corollary}

In the case of stable maps, Harris--Roth--Starr \cite{HRS} prove that for a general hypersurface $X \subset \mbP^{n-1}_\mbC$ of degree at most $n-2$, the Kontsevich moduli space $\overline{\xcal{M}}_{0,m}(\mbP_{\mbC}^1, X, d)$ is a generically smooth, irreducible local complete intersection stack of the expected dimension.
\subsection*{Acknowledgements}
I would like to thank my supervisor Tim Browning for suggesting the problem and for all the useful discussions.
The author is supported by an EPSRC doctoral training grant. 

\section{Preliminaries}\label{Preliminaries}
In this section we establish notation and record some basic definitions and facts. Throughout this paper $S\ll T$ denotes an estimate of the form $S \leq CT$, where $C$ is some constant that does not depend on $q$. Similarly, the implied constants in the notation $S = O(T)$ are independent of $q$.  Let $k = \Fq$, $K= k(t)$, and $\xcal{O} = k[t]$. Finite primes $\varpi$ in $\xcal{O}$ are monic irreducible polynomials and we let $s=t^{-1}$ be the prime at infinity. These have associated absolute values which extend to give absolute values $|\cdot|_\varpi$ and $|\cdot| = |\cdot|_\infty$ on $K$. We let $K_\varpi$ and $K_\infty$ be the completions. We have
$$
	K_\infty = \Fq \left(\left(t^{-1}\right)\right) = \left\{  \sum\limits_{i \leq N} a_i t^i : a_i \in \Fq, N \in \mbZ\right\}.
$$
Set $\mbT = \left\{ \sum_{i \leq -1} a_i t^i | a_i \in \Fq \right\}$. Locally compact topological spaces have Haar measures, hence there is a (Haar) measure on $K_\infty$, and so on $\mbT$. 
This is normalised such that $\int_{\mbT} \diff\alpha = 1$ and is extended to $K_\infty$ in such a way that 
$$
	\int_{ \{\alpha \in K_\infty : \left |\alpha \right | <   \whN  \} } \diff\alpha =  q^N,
$$
for any positive integer $N$. Moreover, this can be extended to $\mbT^n$ and $K_\infty^n$ for any $n \in \mbZ_{>0}$. 
Denote by $\psi : K_\infty \to \mbC^*$ the non-trivial additive character on $K_\infty$, given by
$$
\sum_{i \leq N} a_i t^i \mapsto \exp\left(\frac{2\pi i \Tr_{\Fq/\Fp}(a_{-1})}{p}\right),
$$
where $q$ is a power of $p$.
Throughout this paper, for any real number $R$, let $\whR =q^R$. The following orthogonality property in \cite[Lemma 7]{Kubota} holds.
\begin{lemma} \label{KubotaLemma7} 
For any $N \in \mbZ_{\geq 0}$ and any $\gamma \in K_\infty$, we have 
$$
\sum_{\substack{b \in \xcal{O} \\ |b|<\whN }} \psi(\gamma b) = 
  		\begin{cases}
			\whN, &\text{if } |\gamma|<\whN ^{-1},\\
    			0, &\text{else}.
  		\end{cases}
$$
\end{lemma}
The following lemma corresponds to \cite[Lemma 2.2]{BVBlue} and a proof can also be found in \cite[Lemma 1(f)]{Kubota}. 
\begin{lemma}\label{Lemma2.2inBlue}
Let $Y \in \mbZ$ and $\gamma \in K_\infty$. Then 
$$
\int_{|\alpha|<\whY }\psi(\alpha\gamma)\diff\alpha = 
  		\begin{cases}
			\whY,  &\text{if } |\gamma|<\whY ^{-1},\\
    			0, &\text{otherwise}.
  		\end{cases}
$$
Taking $Y=0$, it follows that
$$
\int_{\alpha \in \mbT}\psi(\alpha\gamma)\diff\alpha = 
  		\begin{cases}
			1, &\text{if } \gamma =0,\\
    			0, &\text{if } \gamma \in \xcal{O} \setminus \{0\}.
  		\end{cases}
$$
\end{lemma}
The next three results are standard, but are proved here since we require versions in which the implied constant is independent of $q$.
\begin{lemma} \label{tau(f)}
Let $\tau(f)$ be the number of monic divisors of a polynomial $f \in \Fq[t]$. Then for any $\varepsilon >0$ there exists a constant $C(\varepsilon)>0$, depending only on $\varepsilon$, such that $\tau(f) \leq C(\varepsilon) |f|^\varepsilon$.
\end{lemma}
\begin{proof}
First note that 
$$
\frac{\tau(f)}{|f|^\varepsilon}  =  \prod_{\varpi^\alpha \mid \mid f} \frac{\alpha +1}{|\varpi^\alpha|^\varepsilon} = \prod_{\substack{\varpi^\alpha \mid \mid f\\ |\varpi| < 2^{1/\varepsilon} }} \frac{\alpha+1}{|\varpi|^{\alpha\varepsilon}}   \prod_{\substack{\varpi^\alpha \mid \mid f\\ |\varpi| \geq 2^{1/\varepsilon} }} \frac{\alpha+1}{|\varpi|^{\alpha\varepsilon}},
$$
where $\varpi$ denotes a prime in $\xcal{O}$. The second factor is less than or equal to 1. 
In the first factor $ |\varpi| < 2^{1/\varepsilon}$, which is equivalent to $d := \deg(\varpi) < \frac{1}{\varepsilon}\frac{\log 2}{\log q} =:D$. Then,
$$
	 \prod_{\substack{\varpi^\alpha \mid \mid f\\ |\varpi| < 2^{1/\varepsilon} }} \frac{\alpha+1}{|\varpi|^{\alpha\varepsilon}} = \prod_{d<D}  \prod_{\substack{\varpi^\alpha \mid \mid f\\ |\varpi| = q^d }} \frac{\alpha+1}{q^{d\alpha\varepsilon}}  \leq \prod_{d<D}  \prod_{\substack{\varpi^\alpha \mid \mid f\\ |\varpi| = q^d }}\left( 1 + \frac{\alpha}{q^{d\alpha\varepsilon}} \right).
$$
Now, if $g(\alpha) = \frac{\alpha}{y^\alpha}$, then $g'(\alpha) =y^{-\alpha}(1 - \alpha \log y)$. Thus, $g$ is maximised at $\alpha = \frac{1}{\log y}$ when $g\left( \frac{1}{\log y}\right) = \frac{1}{e\log y}$. Thus,
$$
	 \prod_{\substack{\varpi^\alpha \mid \mid f\\ |\varpi| < 2^{1/\varepsilon} }} \frac{\alpha+1}{|\varpi|^{\alpha\varepsilon}} \leq \prod_{d<D}  \prod_{|\varpi| = q^d }\left( 1 + \frac{1}{e \log q^{d\varepsilon}} \right) \leq \prod_{d<D}  \left( 1 + \frac{1}{e \log q^{d\varepsilon}} \right)^{\frac{2q^d}{d}},
$$ 
since by \cite[Chapter 2]{Rosen}, the number $a_d$ of primes of degree $d$ satisfies 
\begin{equation} \label{eq: number of primes}
\left |a_d - \frac{q^d}{d} \right| \leq \frac{q^{\frac{d}{2}}}{d} + q^{\frac{d}{3}}.
\end{equation} 
Then, using $1+x \leq e^x$, we obtain
$$
	 \prod_{\substack{\varpi^\alpha \mid \mid f\\ |\varpi| < 2^{1/\varepsilon} }} \frac{\alpha+1}{|\varpi|^{\alpha\varepsilon}} \leq \prod_{d<D}  \left(\exp \left(\frac{1}{e \log q^{d\varepsilon}} \right)\right)^{\frac{2q^d}{d}} = \exp \left ( 2\sum_{d < D} \frac{q^d}{d^2} \cdot\frac{1}{e\varepsilon \log q}  \right).
$$
Now, $ q^d/d^2 $ is increasing with $d$ for $q \geq 4$ and thus, in this case we have $\sum_{d < D} q^d/d^2 < q^D/D$. In fact, we have $\sum_{d < D} q^d/d^2 < 2q^D/D$, for any $q \geq 2$. Thus, 
$$
	 \prod_{\substack{\varpi^\alpha \mid \mid f\\ |\varpi| < 2^{1/\varepsilon} }} \frac{\alpha+1}{|\varpi|^{\alpha\varepsilon}} \leq \exp \left ( \frac{4q^D}{D} \cdot\frac{1}{e\varepsilon \log q}  \right) = \exp \left ( \frac{2^{2+1/\varepsilon}}{e\log 2} \right),
$$
which concludes the proof. 
\end{proof}
\begin{lemma}\label{Lemma: tau_k} 
Let $\omega(f)$ denote the number of prime divisors of a polynomial $f \in \Fq[t]$. Then for any $\varepsilon >0$ and any integer $k \geq 2$, we have $k^{\omega(f)} \ll_{\varepsilon,k} |f|^\varepsilon$.
\end{lemma}
\begin{proof}
Let $\tau_k(f)$ denote the number of factorisations of a polynomial $f \in \Fq[t]$ into $k$ factors. Write $f = \varpi_1^{a_1} \dots \varpi_m^{a_m}$, where $\varpi_i$ are distinct primes in $\Fq[t]$. Then,
$$
\tau_k(f) = \prod_{j=1}^{m}\left(  \begin{matrix} a_j + k - 1\\ a_j \end{matrix} \right) = \prod_{j=1}^{m} \frac{(a_j+k-1) \dots (a_j +1)}{(k-1)!} \geq  \prod_{j=1}^{m} \frac{k! }{(k-1)!} = k^m. 
$$
Thus, $\tau_k(f) \geq k^{\omega(f)}$. We will prove $\tau_k(f) \ll_{\varepsilon} |f|^\varepsilon$, by induction. For $k = 2$, the result follows from Lemma \ref{tau(f)}. 
For $k >2$, use the fact that $\tau_k(f) = \sum_{d \mid f} \tau_{k-1}\left(f/d_k\right)$.
\end{proof}
\begin{lemma}\label{TypicalSum1} Let $Y \in \mbN$. Then
	$$
	\sum_{\substack{ m \in \xcal{O}\\ |m| \leq \whY \\ m \textrm{ monic} }} \frac{1}{|m|} =Y+1.
	$$
\end{lemma}
\begin{proof}
	We have
	$$
		\sum_{\substack{ m \in \xcal{O}\\ |m| \leq \whY \\ m \text{ monic} }} \frac{1}{|m|}  = \sum_{n=0}^Y\frac{1}{q^n} \# \left \{ m \in \xcal{O} : |m| = q^n, m \text{ monic} \right \} = Y+1, 
	$$	
as claimed.  	
\end{proof}
\section{The circle method over function fields}\label{CircleMethod}
Recall that $k=\Fq$ has characteristic $>3$ and $F \in k[x_1, \ldots, x_n]$ denotes a non-singular homogeneous polynomial of degree $3$. Moreover, let $X \subset \mbP_k^{n-1}$ be the smooth cubic hypersurface defined by $F=0$, and let  $\v{a}, \v{b} \in \Fq^n \setminus \{\v{0}\}$ such that $F(\v{a})=F(\v{b})=0$, $ H(\v{b}) \neq 0$. 
We want $x_i \in \Fq[t]$ such that $F(\v{x})=0$ and
	\begin{align}
	\v{x}(0) &=\v{a}, \label{eq:a}\\
	\v{x}(\infty) &=\v{b}. \label{eq:b}
	\end{align}
Now, $t$ is a prime in $\Fq[t]$, and $s=t^{-1}$ is the prime at infinity. Moreover,
$$
x_i	= \sum_{0\leq j \leq d} x_{ij}t^j = t^d\left(\sum_{0\leq j \leq d} x_{ij}t^{j-d} \right) = t^d\left(\sum_{0\leq j \leq d} x_{ij}s^{d-j} \right) = t^dy_i,
$$
say, for $x_{ij} \in \Fq$. Then $y_i=t^{-d}x_i$ and (\ref{eq:a}) is equivalent to $ \v{x} \equiv \v{a} \Mod t$, while (\ref{eq:b}) is equivalent to $\v{y} \equiv \v{b} \Mod s$. 

Define a weight function $\omega : K_\infty^n \to \mbR_{\geq 0}$ such that 
$$ 
\omega(\v{x})  = 
  		\begin{cases}
			1, &\text{if } |t\v{x} -\v{b}|< 1,\\
    			0, &\text{otherwise}.
  		\end{cases}
$$
This is the weight function $w(t^L(\v{x}-\v{x}_0))$ defined in \cite[(7.2)]{BVBlue}, with $\v{x}_0 = t^{-1}\v{b}$ and $L=1$. We now check that \cite[(7.1)]{BVBlue} and \cite[(7.3)]{BVBlue} hold. Recall that $H(\v{x}) = \det \v{H}(\v{x})$ and note that
$F(t^{-1}\v{b}) = t^{-3}F(\v{b}) = 0$, $H(t^{-1}\v{b}) = H(t^{-1}\v{b}) = t^{-n} H(\v{b}) \neq 0,$ and $|t^{-1}\v{b}| =1/q < 1$. 
Furthermore, $|\v{x}| <1$ for any $\v{x} \in K_\infty^n$ such that $\omega(\v{x}) \neq 0$. Moreover, any $\v{x} \in K_\infty^n$ such that $\omega(\v{x}) \neq 0$ can be written in the form $\v{x} = t^{-1}(\v{b} + \v{z})$, where $\v{z}\in \mbT^n$. Then, for any $\v{x} \in K_\infty^n$ such that $\omega(\v{x}) \neq 0$,
$$
|\det \v{H}(\v{x}) | = |H(t^{-1}(\v{b} + \v{z}))| = q^{-n} |H(\v{b}) + \v{z} .\nabla H(\v{b}) + \ldots | =q^{-n},
$$
since $H(\v{b}) \in \Fq^*$ and $\v{z} \in \mbT^n$. But $ |H(t^{-1}\v{b})| = |t^{-n} H(\v{b})| = q^{-n}$, and thus $|H(\v{x}) | = |H(t^{-1}\v{b})|$, for any $\v{x} \in K_\infty^n$ such that $\omega(\v{x}) \neq 0$. This confirms that \cite[(7.1) and (7.3)]{BVBlue} hold.

We have
$$ 
N_{a,b}(d) = \Sum_{\substack{\v{x}\in \xcal{O}^n \\ F(\v{x})=0 \\ \v{x} \equiv \v{a} \Mod t }}\omega\left(\frac{\v{x}}{t^{d+1}}\right),
$$
where $a$ and $b$ are the corresponding points of $\v{a}$ and $\v{b}$ in $X(k)$. We remark that any $\v{x}$ in the sum has $|\v{x}| = q^d$. To simplify notation, we write $N_{a,b}(d) = N(d)$ and $P = t^{d+1}$. Then, $\omega(\v{x}/P)\neq0$ implies that $t^{-d} \v{x} \equiv \v{b} \Mod s$.
Define
$$
S(\alpha) = \Sum_{\substack{\v{x}\in \xcal{O}^n \\ \v{x} \equiv \v{a} \Mod t  }} \psi(\alpha F(\v{x})) \omega \left(\frac{\v{x}}{P}\right).
$$
Then, by Lemma \ref{Lemma2.2inBlue}, we have
$$
N(d) = \int_{\alpha \in \mbT} S(\alpha)\diff \alpha.
$$
By \cite[Lemma 4.1]{BVBlue}, $\mbT$ can be partitioned into a union of intervals centred at rationals and since $K$ is non-archimedean, the intervals do not overlap. 
Thus, for any $Q \geq 1$, we have
\begin{equation}\label{Eq4.1inBlue}
N(d) = \Sum_{\substack{r \in \xcal{O} \\ |r| \leq \whQ  \\ r \text{ monic}}} \quad \Suma \int_{|\theta|<\frac{1}{|r|\whQ }} S\left( \frac{a}{r} + \theta \right) \diff \theta,
\end{equation}
where $\sum\nolimits^*$ denotes a restriction to $(a,r) =1$. We shall take $Q = \frac{3(d+1)}{2}$ in our work.
We now note that $S(a/r + \theta)$ is the same as the exponential sum $S(a/r + \theta)$ appearing in \cite[pg.\ 690]{BVBlue}, where $\v{b}$ is $\v{a}$ and $M=t$. Define $r_M = rM/(r,M)$, for any $M$, 
\begin{align}
S_{r, M, \v{a}}\left( \v{c} \right)= \Suma \Sum_{\substack{\v{y}\in \xcal{O}^n\\ |\v{y}| < |r_M|\\ \v{y} \equiv \v{a} \Mod M}} \psi \left( \frac{aF(\v{y})}{r}\right) \psi\left(\frac{-\v{c}.\v{y}}{r_M}\right), \label{def:SrMa}\\
I_{r}\left(\theta; \v{c}\right) = \int_{K_\infty^n} \omega \left(\v{u}\right) \psi \left( \theta P^3 F\left(\v{u} \right) + \frac{P\v{c}.\v{u}}{r}\right) \diff\v{u}.\label{def:Ir}
\end{align}
\begin{lemma}\label{Lemma4.4inBlue} 
Let $P=t^{d+1}$. We have
\begin{equation}\label{N(d)2}
N(d) = |P|^n  \Sum_{\substack{r \in \xcal{O}\\ |r| \leq \whQ  \\ r \text{ monic}}}  |\tilde{r}|^{-n} \int_{|\theta|<\frac{1}{|r|\whQ }}  \Sum_{\substack{\v{c}\in \xcal{O}^n \\ |\v{c}|\leq \whC}} S_{r, t, \v{a}}\left( \v{c} \right) I_{\tilde{r}}\left(\theta; \v{c}\right) \diff\theta,
\end{equation}
where 
$$
\tilde{r} = \frac{rt}{(r,t)} =
\begin{cases}
rt, &\text{ if } t \nmid r,\\
r, &\text{ otherwise}.
\end{cases}
$$
and $\whC = q |\tilde{r}||P|^{-1} \max\left\{1, |\theta| |P|^3  \right\}$. 
\end{lemma}
\begin{proof}
Applying \cite[(7.7)]{BVBlue} with $\v{x}_0 = t^{-1}\v{b}$ and $L=1$, we have
$$
I_{\tilde{r}}\left(\theta; \v{c}\right) = \frac{1}{q^n}  \psi \left( \frac{P\v{c}.\v{b}}{\tilde{r}t}\right) J_G\left( \theta P^3; \frac{Pt^{-1}\v{c}}{\tilde{r}}\right),
$$
where $G(\v{v}) = F \left( t^{-1}\v{b} + t^{-1}\v{v}\right)$ and 
$$
J_G\left( \theta P^3; \frac{Pt^{-1}\v{c}}{\tilde{r}}\right) = \int_{\mbT^n}\psi (\theta P^3G(\v{x})+\frac{Pt^{-1}\v{c}.\v{x}}{\tilde{r}})\diff \v{x}, 
$$
using the notation in \cite[(2.4)]{BVBlue}. According to \cite[Lemma 2.6]{BVBlue} we have 
$$
J_G\left(\theta P^3 ; \frac{Pt^{-1}\v{c}}{\tilde{r}}\right) = 0
$$
if $|\v{c}| > q |\tilde{r}| |P|^{-1} \max\left\{1, | \theta| |P|^3 \right \} $. Now apply \cite[Lemma 4.4]{BVBlue}.
\end{proof}
We note that $C \in \mbZ$. Our strategy is now to go through the remaining arguments in \cite[Sections 4 -- 9]{BVBlue} for our particular exponential sums and integrals, paying special attention to the uniformity in the $q$-aspect. Furthermore, we keep the same notation as in \cite[Definition 4.6]{BVBlue} for the factorisation of any $r \in \xcal{O}$. Thus, for any $j \in \mbZ_{>0}$ we have $ r = r_{j+1} \prod_{i=1}^{j}b_i = r_{j+1} \prod_{i=1}^j k_i^i$, with $(j+1)$-full $r_{j+1}$, where for any $i \in \mbZ_{>0}$ we have
$$
b_i = \prod_{\varpi^i \mid\mid r} \varpi^i, \quad k_i = \prod_{\varpi^i \mid\mid r} \varpi, \quad r_i = \prod_{\substack{\varpi^e \mid\mid r \\ e \geq i} } \varpi^e.
$$  
\subsection{Exponential Sums}\label{Sum}
We continue to assume that $\ch\left(\Fq\right)>3$. Moreover, we note that $S_{r,M,\v{a}}(\v{c})$ satisfies the multiplicativity property recorded in \cite[Lemma 4.5]{BVBlue}. We are interested in the case when $M \mid t$.
\begin{lemma}\label{Lemma4.5inBlue} 
Let $r=uv$ for coprime $u,v\in \xcal{O}$ and $t \nmid u$. Then there exist non-zero $\v{a}',\v{a}''\in k^n$, depending on $\v{a}$ and the residues of $u,v$ modulo $t$, such that 
$$
S_{r,M,\v{a}}(\v{c})
=
\begin{cases}
S_{u,1,\v{0}}(\v{c}) S_{v,1,\v{0}}(\v{c}), &\text{if } M=1,\\
S_{u,1,\v{0}}(\v{c}) S_{v,t,\v{a}'}(\v{c}), &\text{if } M=t \text{ and } t\mid r,\\
S_{u,1,\v{0}}(\v{c}) S_{v,1,\v{0}}(\v{c})\psi(\frac{-\v{c}.\v{a}''}{t}), &\text{if } M=t \text{ and } t\nmid r.
\end{cases}
$$
\end{lemma}

 Furthermore, the estimates in \cite[Lemma 5.1]{BVBlue}, \cite[(5.2)]{BVBlue} and \cite[(5.3)]{BVBlue} all hold and are independent of $q$. Next we record the following result. 
\begin{lemma}\label{Lemma6.4inBlue}
	Let $\v{r} \in K_\infty^n$, $C \in \mbN$, $M \in \xcal{O}$ and $\varepsilon >0$. Then there exists a constant $c_{n, \varepsilon} >0$, depending only on $n$ and $\varepsilon$, such that
	$$
	\Sum_{\substack{\v{c} \in \xcal{O}^n \\ |\v{c} - \v{r}|<\whC}} |S_{r_3, M, \v{a}}(\v{c})| \leq c_{n, \varepsilon} |M|^n |r_3|^{n/2+1+\varepsilon}\left( |r_3|^{n/3} + \whC^n \right).
	$$
\end{lemma}
\begin{proof}
This follows directly from \cite[Lemma 6.4]{BVBlue} on noting that $H_F=|\Delta_F|=1$ in our situation. 
\end{proof}
Let $F^* \in k[x_1, \ldots, x_n]$ be the dual form of $F$. Its zero locus parametrises the set of hyperplanes whose intersection with the cubic hypersurface $F=0$ produces a singular variety. Moreover, $F^*$ is absolutely irreducible and has degree $3\cdot 2^{n-2}$. 
We shall need the following variation of \cite[Lemma 6.4]{BVBlue} in which the sum is restricted to zeros of $F^*$.
\begin{lemma}\label{ModifiedLemma6.4inBlue}
	Let $C \in \mbN$, $M \in \xcal{O}$ and $\varepsilon >0$. Then there exists a constant $c_{n, \varepsilon} >0$, depending only on $n$ and $\varepsilon$, such that
	$$
	\Sum_{\substack{\v{c} \in \xcal{O}^n \\ |\v{c}|<\whC \\ F^*(\v{c}) = 0}} |S_{r_3, M, \v{a}}(\v{c})| \leq c_{n, \varepsilon} |M|^{\frac{2n-5}{3}  +\varepsilon}   \left |r_3 \right|^{\frac{5n + 7}{6}  +\varepsilon }\left( 1+ \whC \right)^{n-\frac{3}{2}+\varepsilon}.
	$$
\end{lemma}
\begin{proof}
This proof uses the same methods as in Section 7 of \cite{Hea83}. By (\ref{def:SrMa}), we have
$$S_{r_3, M, \v{a}}(\v{c})  =  {\mathlarger{\sideset{}{^*}\sum}_{|a|<|r_3|}} \Sum_{\substack{\v{y}\in \xcal{O}^n\\ |\v{y}| < |r_{3}|\\ \v{y} \equiv \v{a} \Mod M}} \psi \left( \frac{aF(\v{y})-\v{c}.\v{y}}{r_3}\right),$$
since $r_{3_{M}} = r_3 M/(r_3,M) = r_3$.
Setting $\v{y} = \v{a} + M \v{z}$, we have 
$$
S_{r_3, M, \v{a}}(\v{c})  =  \psi\left(\frac{-\v{c}.\v{a}}{r_3}\right) {\mathlarger{\sideset{}{^*}\sum}_{|a|<|r_3|}}  S_a(\v{c}),
$$
where 
$$
S_a(\v{c}) = \Sum_{\substack{\v{z}\in \xcal{O}^n\\ |\v{z}| < |s|}} \psi \left( \frac{aM^{-1}F(\v{a}+M\v{z}) - \v{c}.\v{z}}{l}\right)
$$
and $l = r_3/M$.
Denote 
\begin{equation}\label{eq:g(z)}
g(\v{z}) = M^{-1}F(\v{a}+M\v{z}), 
\end{equation}
write $l=c^2d$, where $d$ is square-free, $d \mid c$, and put $\v{z} = \v{z}_1 + cd\v{z}_2$, with $|\v{z}_1| < |cd|$. Then,
\begin{align*}
S_a(\v{c})     & = |c|^n  \Sum_{\substack{\v{z}_1\in \xcal{O}^n\\ |\v{z}_1| < |cd| \\ a \nabla g(\v{z}_1) \equiv \v{c} \Mod{c} }}  \psi \left( \frac{ag(\v{z}_1) - \v{c}.\v{z}_1}{c^2d}\right).
\end{align*} 
Write $a = a_1 + Mca_2$ with $|a_1| < |Mc|$. Then $(a,r_3) = 1$ if and only if $(a_1, Mc) =1$ and thus,
\begin{align*}
{\mathlarger{\sideset{}{^*}\sum}_{|a|<|r_3|}}  S_a(\v{c}) 
	& = |c|^n |cd|  {\mathlarger{\sideset{}{^*}\sum}_{|a_1|<|Mc|}}\Sum_{\substack{\v{z}_1\in \xcal{O}^n\\ |\v{z}_1| < |cd| \\ a_1 \nabla g(\v{z}_1) \equiv \v{c} \Mod{c} \\ g(\v{z}_1) \equiv 0 \Mod{cd} }}  \psi \left( \frac{a_1g(\v{z}_1) - \v{c}.\v{z}_1}{c^2d}\right).
\end{align*}

Writing $\v{z}_1 = \v{h} + c\v{j}$ with $|\v{h}| < |c|$, we have $g(\v{z}_1) \equiv g(\v{h}) + c\v{j} \nabla g(\v{h}) \Mod{cd}$ and $a_1 \nabla g(\v{z}_1) \equiv a_1 \nabla g(\v{h}) \Mod{c}$, since $cd \mid c^2$.
Now, $g(\v{z}_1) \equiv 0 \Mod {cd}$ is equivalent to $g(\v{h}) + c\v{j} \nabla g(\v{h})  \equiv 0 \Mod {cd} $. 
Thus, $g(\v{h})  \equiv 0 \Mod {c} $ and we can write $g(\v{h}) = mc$. Thus, $m + \v{j} \nabla g(\v{h}) \equiv 0 \Mod{d}$. Moreover, if  $a_1 \nabla g(\v{h}) = \v{c} + c\v{k}$, then
$
a_1g(\v{z}_1) - \v{c}.\v{z}_1 \equiv a_1g(\v{h}) - \v{c}.\v{h} + c^2(\v{k}.\v{j} + a_1\v{h}\nabla g(\v{j})) \Mod{c^2d}, 
$
and thus, the sum over $\v{z}_1$ becomes
$$
\Sum_{\substack{\v{h}\in \xcal{O}^n\\ |\v{h}| < |c| \\ a_1 \nabla g(\v{h}) = \v{c} + c\v{k} \\ g(\v{h}) = mc }} \psi \left( \frac{a_1g(\v{h}) - \v{c}.\v{h}}{c^2d}\right)\Sum_{\substack{\v{j}\in \xcal{O}^n\\ |\v{j}| < |d| \\ m + \v{j} \nabla g(\v{h})  \equiv 0 \Mod{d} }} \psi \left( \frac{\v{k}.\v{j} + a_1\v{h}\nabla g(\v{j}) }{d}\right).
$$

Denote the sum over $\v{j}$ by $S_{\v{k}, \v{h}}$ and estimate it by writing
$$
\left | S_{\v{k}, \v{h}} \right|^2 = \Sum_{\substack{\v{j}_1, \v{j}_2 \in \xcal{O}^n\\ |\v{j}_1|, |\v{j}_2| < |d| \\ \v{j}_1 \nabla g(\v{h})  \equiv \v{j}_2 \nabla g(\v{h}) \Mod{d} }} \psi \left( \frac{\v{k}.(\v{j}_1 - \v{j}_2) + a_1\v{h}(\nabla g(\v{j}_1) - \nabla g(\v{j}_2))}{d}\right).
$$
Writing $\v{j}_1 = \v{j}_2 + \v{j}_3$ and recalling (\ref{eq:g(z)}), we note that
$$
\v{h} \left( \nabla g(\v{j}_1) - \nabla g(\v{j}_2) \right) =  \frac{1}{2}\v{j}_3^T\nabla^2 g(\v{h})\v{j}_3 + \v{j}_2^T\nabla^2 g(\v{h})\v{j}_3
$$ 
and therefore,
\begin{align*}
\left | S_{\v{k}, \v{h}} \right|^2 
	& \leq \Sum_{\substack{\v{j}_3\in \xcal{O}^n\\ |\v{j}_3| < |d| \\ \v{j}_3 \nabla g(\v{h})  \equiv 0 \Mod{d} }} \left | \Sum_{\substack{\v{j}_2 \in \xcal{O}^n\\ |\v{j}_2| < |d| }} \psi \left( \frac{ \v{j}_2 .(a_1\nabla^2 g(\v{h})\v{j}_3)}{d}\right) \right| 
	 \leq |d|^n M_d(\v{h}),	
\end{align*}
where
$
M_d(\v{h}) = \# \left\{\v{j}_3\in \xcal{O}^n : |\v{j}_3| < |d|,  \nabla^2 g(\v{h})\v{j}_3 \equiv 0 \Mod{d} \right\}.
$
Thus,
\begin{align*}
\Sum_{\substack{\v{c} \in \xcal{O}^n \\ |\v{c}|<\whC \\ F^*(\v{c}) = 0}} |S_{r_3, M, \v{a}}(\v{c})| 
	& \leq |c|^{n+2} |d|^{n/2+1} |M| \Sum_{\substack{\v{c} \in \xcal{O}^n \\ |\v{c}|<\whC \\ F^*(\v{c}) = 0}}  \quad \Sum_{\substack{\v{h}\in \xcal{O}^n\\ |\v{h}| < |c| \\  g(\v{h}) \equiv 0 \Mod{c} }} M_{d}(\v{x})^{1/2}.
\end{align*}

Now note that there exist elements $c' = cM$ and $d' =\frac{d}{(d,M)}$ with $d'\mid c'$, such that
$$
\Sum_{\substack{\v{h}\in \xcal{O}^n\\ |\v{h}| < |c| \\  g(\v{h}) \equiv 0 \Mod{c} }} M_d(\v{h})^{1/2} \leq |(d, M)|^{n/2} \Sum_{\substack{\v{x}\in \xcal{O}^n\\ |\v{x}| < |c'| \\ F(\v{x}) \equiv 0 \Mod{c'} }} N_{d'}(\v{x})^{1/2},
$$
where $ \v{x} = \v{a}+ M \v{h}$ and $N_{d'}(\v{x}) = \# \left\{\v{y} \in \xcal{O}^n : |\v{y} | < |d'|,  \v{H}(\v{x}) \v{y} \equiv 0 \Mod{d'} \right\}.$ Thus,
\begin{align*}
\Sum_{\substack{\v{c} \in \xcal{O}^n \\ |\v{c}|<\whC \\ F^*(\v{c}) = 0}} |S_{r_3, M, \v{a}}(\v{c})| & \leq  |c|^{n+2} |d|^{n/2+1} |M| |(d, M)|^{n/2}  \xcal{N} \Sum_{\substack{\v{x}\in \xcal{O}^n\\ |\v{x}| < |c'| \\  F(\v{x}) \equiv 0 \Mod{c'} }} N_{d'}(\v{x})^{1/2},
\end{align*}
where
$
\xcal{N} := \# \left\{\v{c} \in \xcal{O}^n : |\v{c}|<\whC, F^*(\v{c}) = 0 \right\} \ll \left( 1+ \whC \right)^{n-3/2+\varepsilon} 
$ 
for any $\varepsilon > 0$, by \cite[Lemma 2.10]{BVBlue}.

It remains to bound the inner sum. As is \cite{BVBlue}, let
$$
S(c,d) = \Sum_{\substack{\v{x}\in \xcal{O}^n\\ |\v{x}| < |c| \\  F(\v{x}) \equiv 0 \Mod{c} }} N_{d}(\v{x})^{1/2},
$$
for given $c, d$ in $\xcal{O}$, where $d \mid c$ and $d$ is square-free. This sum satisfies a multiplicativity property, i.e. for any $c_i, d_i$ in $\xcal{O}$ such that $(c_1d_1, c_2d_2) =1$ and $d_i \mid c_i$ we have $S(c_1c_2, d_1d_2) = S(c_1,d_1)S(c_2,d_2).$ Thus, we only need to look at the cases when $c = \varpi^e$ and $d=1$, and $c = \varpi^e$ and $d= \varpi$, for any $e \in \mbZ_{>0}$ and any prime $\varpi$. Note that $F$ is non-singular modulo any prime $\varpi$. 

The arguments that follow are similar to \cite[p. 244]{Hea83}. Define
\begin{align*}
S_0(\varpi^e) &=  \# \left \{ \v{x} \in \xcal{O}^n: |\v{x}| < |\varpi|^e, F(\v{x}) \equiv 0 \Mod{\varpi^e} \right\},\\
S_1(\varpi^e) &=  \# \left \{ \v{x} \in \xcal{O}^n: |\v{x}| < |\varpi|^e, \varpi \nmid \v{x}, F(\v{x}) \equiv 0 \Mod{\varpi^e} \right\},
\end{align*}
for $e \geq 1$. Then, as in \cite[(7.4), (7.5)]{Hea83}, we have
\begin{align}
S_0(\varpi^e) &= S_1(\varpi^e) + |\varpi|^{2n} S_0(\varpi^{e-3}), \text{ for } e \geq 4,\label{eq:S0e>3}\\
S_0(\varpi^e) &= S_1(\varpi^e)  + |\varpi|^{(e-1)n}, \text{for } 1\leq e \leq 3,\label{eq:S0e<4}\\
S_1(\varpi^{e+1}) &= |\varpi|^{n-1}S_1(\varpi^e), \text{for } e\geq 1. \label{eq: S1}
\end{align}
Since, $N_1(\v{x}) = 1$, we have $S(\varpi^e, 1)  = S_0(\varpi^e).$ Moreover, $S_1(\varpi) \ll |\varpi|^{n-1}$, and thus, 
\begin{equation}
S_1(\varpi^e) \ll |\varpi|^{e(n-1)}, \label{eq: S1 bound}
\end{equation}
for $e \geq 1$.  Thus, for $1 \leq e \leq 3$ and $n \geq 4$, we have $S_0(\varpi^e) \ll |\varpi|^{e(n-1)}.$
Similarly, for $e \geq 4$ and $n \geq 4$, we can use an induction argument to get $S_0(\varpi^e) \ll |\varpi|^{e(n-1)}$. Thus, for $c = \varpi^e$ and $d = 1$, $S(c,d) \leq A_1^{\omega(c)}|c|^{n-1}$. 
Consider now the case when $c = \varpi^ e$ and $d = \varpi$. After a change of variables, $S(\varpi ^e,\varpi)$ is equal to
\begin{align*}
	& \Sum_{\substack{\v{z}\in \xcal{O}^n\\ |\v{z}| < |\varpi | }} N_{\varpi}(\v{z})^{1/2} \# \left \{ \v{x} \in \xcal{O}^n : |\v{x}|< |\varpi| ^e , \v{x} \equiv \v{z} \Mod{\varpi}, F(\v{x}) \equiv 0 \Mod{\varpi ^e} \right \}. 
\end{align*}

First analyse the contribution to $S(\varpi ^e,\varpi)$ coming from $\v{z}$ such that $\varpi \nmid \v{z}$. Then, as in \cite{Hea83}, by Cauchy's inequality, it follows that this contribution is
\begin{align*}
&\ll \varpi^{(e-1)(n-1)}  \Sum_{\substack{\v{z}\in \xcal{O}^n\\ |\v{z}| < |\varpi | \\ \varpi \nmid \v{z} \\ F(\v{z}) \equiv 0 \Mod{\varpi} }} N_{\varpi}(\v{z})^{1/2}
\leq \varpi^{(e-1)(n-1)} S_{N_\varpi}(\v{z})^{1/2} S_0(\varpi)^{1/2},
\end{align*}
where 
$$
S_{N_\varpi}(\v{z}) = \# \left \{ \v{z}, \v{y} \in \xcal{O}^n : |\v{z}| < |\varpi |, |\v{y}| < |\varpi |, \varpi \nmid \v{z}, \v{H}(\v{z}).\v{y}  \equiv \v{0} \Mod {\varpi } \right \}.
$$
Then, by (\ref{eq: S1}), (\ref{eq: S1 bound}) and \cite[Lemma 4]{Hea83}, there exists some constant $A$ such that the contribution to $S(\varpi^e, \varpi)$ coming from $\v{z}$ such that $\varpi \nmid \v{z}$ is
\begin{align*}
& \leq  A^{\omega(\varpi)} \varpi^{(e-1)(n-1)} |\varpi|^{\frac{n}{2}} |\varpi|^{\frac{n-1}{2}} =  A \varpi^{e(n-1) + \frac{1}{2}}.
\end{align*}

The remaining contribution to $S(\varpi ^e,\varpi)$ comes from $\v{z} = \v{0}$. In this case,
$
N_{\varpi}(\v{0}) = |\varpi|^n, 
$
and thus this contribution is 
\begin{align}
& |\varpi |^{\frac{n}{2}}  \# \left \{ \v{y} \in \xcal{O}^n : |\v{y}|< |\varpi| ^{e-1} , \varpi^3F(\v{y}) \equiv 0 \Mod{\varpi ^e} \right \}.\label{eq: contrib from z = 0}
\end{align}
Then, as in \cite{Hea83}, if $1 \leq e \leq 3$, (\ref{eq: contrib from z = 0}) becomes
\begin{equation} \label{eq: e < 3}
|\varpi |^{\frac{n}{2}}  \# \left \{ \v{y} \in \xcal{O}^n : |\v{y}|< |\varpi| ^{e-1} \right \} = |\varpi |^{n(e-1/2)}, 
\end{equation}
and if $4 \leq e$, there exists a constant $A$ such that (\ref{eq: contrib from z = 0}) is equal to
\begin{align}
& |\varpi |^{\frac{5n}{2}} S_0(\varpi^{e-3}) \leq A^{\omega(\varpi)} |\varpi |^{\frac{5n}{2} + (e-3)(n-1)} = A |\varpi |^{e(n-1) - \frac{n}{2} + 3}. \label{eq: e > 3}
\end{align}
Note that if $n \geq 5$, then the contributions in (\ref{eq: e < 3}) and (\ref{eq: e > 3}) are both $\ll |\varpi|^{e(n-1) + \frac{1}{2}}$, and thus  $S(c,d) \ll A^{\omega(c)} |c|^{n-1}|d|^{1/2}$.

Putting everything together, we have 
\begin{align*}
\Sum_{\substack{\v{c} \in \xcal{O}^n \\ |\v{c}|<\whC \\ F^*(\v{c}) = 0}} |S_{r_3, M, \v{a}}(\v{c})| & \leq  |c|^{n+2} |d|^{\frac{n}{2}+1} |M| |(d, M)|^{\frac{n}{2}} A^{\omega(c')} |c'|^{n-1}|d'|^{\frac{1}{2}} \left( 1+ \whC \right)^{n-\frac{3}{2}+\varepsilon}.
\end{align*}
Then, an application of Lemma \ref{Lemma: tau_k} concludes the proof.
\end{proof}
\subsection{Exponential Integral}\label{Integral}
The following result is similar to \cite[Lemma 7.3]{BVBlue}. It gives a good upper bound for $I_{\tilde{r}}\left(\theta; \v{c}\right)$, for $r, \theta, \v{c}$ appearing in the expression for $N(d)$  in Lemma \ref{Lemma4.4inBlue}.
\begin{lemma}\label{Lemma7.3inBlue}
We have
$$
\left |I_{\tilde{r}}\left(\theta; \v{c}\right) \right| \ll \min \left\{q^{-n}, q^n |\theta P^3|^{-n/2} \right\},
$$
where the implicit constant is independent of $q$.
\end{lemma}
\begin{proof}
As in \cite[Lemma 7.3]{BVBlue}, we have $\left |I_{\tilde{r}}\left(\theta; \v{c}\right) \right| \leq \meas (\xcal{R})$, where 
$$\xcal{R} = \left\{ \v{x} \in \mbT^n : |t\v{x} - \v{b}| < 1, |\theta P^3 \nabla F(\v{x}) + Pt^{-1}\v{c}/\tilde{r} | \leq \max \left\{1, |\theta P^3|^{1/2} \right \} \right \}.$$ 
If $|\theta P^3| \leq 1$, then we have the trivial bound $\meas ( {\xcal{R}} ) \leq q^{-n}$. Otherwise, given $\v{x} \in \xcal{R}$, we can write it as $\v{x} = \v{b}t^{-1} + \v{d}$, where $\v{d} \in \mbT^n$, $|\v{d}| \leq q^{-2}$. Then
$$
|H(\v{x})| = |t^{-n}H(\v{b}) + \v{d}.\nabla H(\v{b}t^{-1}) + \ldots | = q^{-n}( 1 + O(q^{-1})).
$$
Since the entries in the adjugate of $\v{H}(\v{x})$ have norms equal to $q^{-n+1}(1+O(q^{-1}))$, the inverse of $\v{H}(\v{x})$ has entries with absolute value $q + O(1)$. Thus, if $\v{x}$ and $\v{x} + \v{x}'$ are in $\xcal{R}$, we have $|\v{x'}| \ll q |\theta P^3|^{-1/2}$ and thus, $\meas (\xcal{R}) \ll q^n  |\theta P^3|^{-n/2} $.
\end{proof}
Note that this result uses crucially the condition that one of the two fixed points in Theorem \ref{MainResult} does not lie on the Hessian of $X$.
\section{The main term}\label{Main Term}
In this section we investigate the contribution to $N(d)$ in Lemma \ref{Lemma4.4inBlue} coming from $\v{c}=\v{0}$. Preserving the notation in \cite{BVBlue}, denote this term by $M(d)$. We will always assume $n \geq 10$. 
Thus, 
\begin{align*}
M(d) 	&= |P|^n  \Sum_{\substack{r \in \xcal{O}\\ |r| \leq \whQ  \\ r \text{ monic}}}  |\tilde{r}|^{-n} S_{r, t, \v{a}}\left( \v{0} \right) K_r,
\end{align*}
where
\begin{align*}
K_r &= \int_{|\theta|<\frac{1}{|r|\whQ }}  I_{\tilde{r}}\left(\theta; \v{0}\right) \diff\theta.\end{align*}
Recall that $\nabla F(t^{-1}\v{b}) \neq \v{0}$ and, in particular, $q^{-2} = |\nabla F(t^{-1}\v{b})|.$ This corresponds to taking $\xi = -2$ in \cite[Section 7.3]{BVBlue}. The following result gives a similar bound to that in \cite[Lemma 7.4]{BVBlue}.
\begin{lemma} \label{Lemma7.4inBlue} 
For any $Y  \in \mbN$ and any $\varepsilon >0$ we have
$$
\Sum_{\substack{r \in \xcal{O}\\ |r| = \whY  \\ r \text{ monic}}}  |\tilde{r}|^{-n}  \left |S_{r, t, \v{a}}\left( \v{0} \right) \right | \ll q^{2n} \whY^{-\frac{n}{6}+ \frac{4}{3} + \varepsilon} (q^{-n} + \whY^{\frac{3-n}{3}} ),
$$
where the implicit constant is independent of $q$.
\end{lemma}
\begin{proof}
Write $r=b_1b_2r_3$. Then, by the multiplicativity property in Lemma \ref{Lemma4.5inBlue}, we have
$$
\left |S_{r, t, \v{a}}\left( \v{0} \right) \right | = \left |S_{b_1b_2, M, \v{a}'}\left( \v{0} \right) \right | \left |S_{r_3, M_3, \v{a}''}\left( \v{0} \right) \right |,
$$
where $M, M_3 \in \{ 1, t\}$ and $\v{a}', \v{a}'' \in k^n$ depend on $\ord_t(r)$. By \cite[Lemma 5.1]{BVBlue},
$$
\left |S_{b_1b_2, M, \v{a}}\left( \v{0} \right) \right | \ll \left | b_1b_2 \right|^{\frac{n}{2}+1+\varepsilon},
$$
where the implicit constant is independent of $q$. Moreover,
$$
\left |S_{r_3, M_3, \v{a}}\left( \v{0} \right) \right | \leq \Sum_{\substack{\v{c} \in \xcal{O}^n \\ |\v{c}| < \whC }}\left | S_{r_3,M_3,\v{a}}(\v{c}) \right|,
$$
for any $C >0$. Taking $C=1$ and using \cite[Lemma 6.4]{BVBlue}, we get
$$
\left |S_{r_3, M_3, \v{a}}\left( \v{0} \right) \right | \leq \Sum_{\substack{\v{c} \in \xcal{O}^n \\ |\v{c}| < q }}\left | S_{r_3,M_3,\v{a}}(\v{c}) \right| \ll |M_3|^n |r_3|^{n/2 + 1 + \varepsilon} \left( |r_3|^{n/3} + q^n \right), 
$$
where the  the implicit constant depends only on $n$ and $\varepsilon$. On noting that for $|r| = \whY$ we have $|\tilde{r}|^{-n} = q^{-n}\whY^{-n}$ if $t\nmid r$ and $|\tilde{r}|^{-n} = \whY^{-n}$, otherwise, we obtain
\begin{align*}
\Sum_{\substack{r \in \xcal{O}\\ |r| = \whY  \\ r \text{ monic}}}  |\tilde{r}|^{-n}  \left |S_{r, t, \v{a}}\left( \v{0} \right) \right | &\ll |M_3|^n \whY^{-\frac{n}{2}+1+\varepsilon} \Sum_{\substack{r_3 \in \xcal{O}\\ |r_3| \leq  \whY  \\ r_3 \text{ monic} }}  \left( |r_3|^{n/3} + q^n \right) \frac{\whY}{|r_3|}.
\end{align*} 
Then, since $\# \left \{r_3 \in \xcal{O} : |r_3| \leq \whY \right \} = O(\whY^{1/3})$ and $M_3 \in \{1,t\}$, we can bound the above by
\begin{align*}
&\ll q^n \whY^{-\frac{n}{2}+\frac{7}{3}+\varepsilon}\left( \whY^{\frac{n}{3}-1} + q^n \right),
\end{align*}
which concludes the proof.
\end{proof}
Put $C= \widehat{L-\xi} = q^3$. Then, if $C^{-1}\whQ  \leq |r| \leq \whQ $, we have $|\theta| < |r|^{-1} \whQ^{-1} \leq q^3 |P|^{-3}$, and thus, $|\theta P^3| \leq q^2$. Then, by Lemma \ref{Lemma7.3inBlue} we have $K_r = O(q^{3-n}|P|^{-3})$ in this case. On noting that the exponents of $\whY$ in the bound given by Lemma \ref{Lemma7.4inBlue} are negative for $n>8$, we obtain that 
\begin{align*}
	\Sum_{\substack{r \in \xcal{O}\\ q^{-3}\whQ \leq |r| \leq \whQ  \\ r \text{ monic}}}  |P|^n  |\tilde{r}|^{-n} S_{r, t, \v{a}}\left( \v{0} \right) K_r &\ll |P|^{n-3} q^{3+n}q^{(Q-3)(-\frac{n}{6}+ \frac{4}{3} + \varepsilon)} (q^{-n} + q^{(Q-3)\frac{3-n}{3}} ).
\end{align*}

Thus, recalling that $\whQ^2 = |P|^3$, the contribution to $M(d)$ coming from such $r$ is $\ll q^{11n/12+2/3}|P|^{3n/4-1+\varepsilon'}\left(1 + q^{2n-3}|P|^{3/2 -n/2} \right)$.
If $|r| < q^{-3}\whQ$, as in \cite[Section 7.3]{BVBlue}, $K_r$ is independent of $r$. Moreover, we only get a contribution from $|\theta | < q^3 |P|^{-3}$. Thus, for $d \geq 3(n-1)/(n-3)$, we have
\begin{align*}
M(d) 	&= |P|^{n-3} \mathfrak{S}(Q) \mathfrak{J} + O(q^{11n/12+2/3}|P|^{3n/4-1}),
\end{align*}
where
$$
 \mathfrak{S}(Q)  =  \Sum_{\substack{r \in \xcal{O}\\ |r| \leq \whQ  \\ r \text{ monic}}}  |\tilde{r}|^{-n} S_{r, t, \v{a}}\left( \v{0} \right)
$$
and
$$ 
 \mathfrak{J}  =  \int_{|\varphi|<q^3} \int_{K_\infty^n} w \left(t\v{u} - \v{b}\right) \psi \left(\varphi F\left(\v{x}\right)\right) \diff\v{x} \diff\varphi
$$
is the \emph{singular integral}. By taking $\v{x}_0 = t^{-1}\v{b}$ and $L=1$ in \cite[Lemma 7.5]{BVBlue}, it follows that
\begin{equation}\label{Lemma7.5inBlue}
 \mathfrak{J} = \frac{1}{q^{n-3}}. 
\end{equation}

By Lemma \ref{Lemma7.4inBlue} we can extend the summation over $r$ in $\mathfrak{S}(Q)$ to infinity with acceptable error since 
\begin{align*}
 \mathfrak{S} -  \mathfrak{S}(Q) 
 	& \ll q^{\frac{5n}{6}+\frac{4}{3}}|P|^{-\frac{n}{4}+2+\varepsilon}\left(1+q^{\frac{2n+3}{3}}|P|^{2(3-n)}\right).
\end{align*} 
and thus, 
$$
|P|^{n-3}\left(  \mathfrak{S} -  \mathfrak{S}(Q) \right ) \mathfrak{J} = |P|^{\frac{3n}{4}-1} q^{-\frac{n}{6} + \frac{13}{3}} \left(1+q^{\frac{2n+3}{3}}|P|^{2(3-n)} \right). 
$$
Then   
$$
 \mathfrak{S} =  \Sum_{\substack{r \in \xcal{O}\\ r \text{ monic}}}  |\tilde{r}|^{-n} S_{r, t, \v{a}}\left( \v{0} \right)
 $$
 is the absolutely convergent \emph{singular series}.
\begin{lemma}\label{Lemma: singular series}
We have
$$
 \mathfrak{S} = q^{-n+1} + O(q^{-3n/2+3}).
 $$
\end{lemma}
\begin{proof}
First, recalling the definition of $\tilde{r}$, decompose $\mathfrak{S}$ into
\begin{equation}\label{eq: decomp frakS}
 \mathfrak{S} = q^{-n} \Sum_{\substack{r \in \xcal{O}\\ r \text{ monic} \\ (r,t)=1}}  |r|^{-n} S_{r, t, \v{a}}\left( \v{0} \right) + \Sum_{\substack{r \in \xcal{O}\\ r \text{ monic} \\ t \mid r}}  |r|^{-n} S_{r, t, \v{a}}\left( \v{0} \right).
\end{equation}
Then note that by the multiplicativity property in Lemma \ref{Lemma4.5inBlue}, given $r = t^A \prod_{\varpi \neq t} \varpi^e \in \xcal{O}$, where $A \in \mbZ_{\geq 0}$ and $\varpi$ are primes in $\xcal{O}$, we have
$$
S_{r,t,\v{a}}(\v{0}) = S_{t^A, t, \v{a}}(\v{0}) \prod_{\substack{\varpi \text{ prime} \\ \varpi \neq t\\ \varpi^e \mid \mid r}} S_{\varpi^e}(\v{0}),
$$
where $S_{\varpi^e}(\v{0}) = S_{\varpi^e, 1, \v{0}}(\v{0}) = S_{\varpi^e, 1, \v{a}}(\v{0})$. Thus, 
$$
 \mathfrak{S} = \left(q^{-n} + \sum_{A=1}^\infty q^{-An} S_{t^A,t,\v{a}}(\v{0}) \right) \prod_{\substack{ \varpi \text{ prime} \\ \varpi \neq t}} \sum_{e=0}^\infty |\varpi|^{-en} S_{\varpi^e}(\v{0}).
$$

Now, by (\ref{def:SrMa}), 
$$
S_{t,t,\v{a}}(\v{0}) =  \mathlarger{\sideset{}{^*}\sum}_{|a|<|t|} \Sum_{\substack{\v{y}\in \xcal{O}^n\\ |\v{y}|<|t| \\ \v{y} \equiv \v{a} \Mod{t} }} \psi \left( \frac{aF(\v{y})}{t}\right) = \mathlarger{\sideset{}{^*}\sum}_{|a|<|t|} \psi \left( \frac{aF(\v{a})}{t}\right) = \mathlarger{\sideset{}{^*}\sum}_{|a|<|t|} 1 = q -1,
$$
since $F(\v{a})=0$. Similarly, by (\ref{def:SrMa}), after making the change of variables $\v{y} = \v{a} + t\v{z}$, we have
$$
S_{t^2,t,\v{a}}(\v{0}) 
	=  \mathlarger{\sideset{}{^*}\sum}_{|a|<|t|^2} \Sum_{\substack{\v{z}\in \xcal{O}^n\\ |\v{z}|<|t| }} \psi \left( \frac{a\v{z}.\nabla F(\v{a})}{t}\right)
	= \mathlarger{\sideset{}{^*}\sum}_{\substack{|a|<|t|^2 \\ a\nabla F(\v{a}) \equiv 0 \Mod{t}}} q^n = 0,
$$
Moreover, for $K \geq 3$, we have
\begin{align*}
S_{t^K,t,\v{a}}(\v{0}) 
	&= \Sum_{\substack{\v{y}\in \xcal{O}^n\\ |\v{y}|<|t|^K \\ \v{y} \equiv \v{a} \Mod{t} }} \left( \mathlarger{\sideset{}{}\sum}_{|a_1|<|t|^K} \psi \left( \frac{a_1F(\v{y})}{t^K}\right) -  \mathlarger{\sideset{}{}\sum}_{|a_2|<|t|^{K-1}} \psi \left( \frac{a_2F(\v{y})}{t^{K-1}}\right) \right)\\
	&= q^K S_{\v{a}}(K) - q^{K-1 + n} S_{\v{a}}(K-1),
\end{align*}
where 
$
S_{\v{a}}(K)  = \# \left \{ \v{y} \in \xcal{O}^n : |\v{y}| < |t|^K, \v{y} \equiv \v{a} \Mod{t}, F(\v{y}) \equiv 0 \Mod{t^{K}} \right \}.
$
Similarly to \cite[p. 244]{Hea83}, we have $S_{\v{a}}(K)  = q^{n-1}S_{\v{a}}(K-1)$. Thus, for $K \geq 3$, $S_{t^K,t,\v{a}}(\v{0}) = 0$. 

It remains to analyse $S_{\varpi^e}(\v{0})$. By (\ref{def:SrMa}), we have $S_1(\v{0}) = 1$. Moreover, by \cite[(5.2), (5.3)]{BVBlue}, we have $S_{\varpi}(\v{0}) \ll |\varpi|^{\frac{n}{2}+1}$ and $S_{\varpi^2}(\v{0}) \ll |\varpi|^{n+2}$. Also, \cite[Lemma 5.3]{BVBlue} implies that $S_{\varpi^3}(\v{0}) \ll |\varpi|^{2n+3}$ and $S_{\varpi^4}(\v{0}) \ll |\varpi|^{3n+3}$. By similar arguments as above, for $e \geq 5$ we have
$$
S_{\varpi^e}(\v{0}) = |\varpi|^e S_0(\varpi^e) - |\varpi|^{e-1+n}S_0(\varpi^{e-1}),
$$
where $S_0(\varpi^e) =  \# \left \{ \v{x} \in \xcal{O}^n: |\v{x}| < |\varpi|^e, F(\v{x}) \equiv 0 \Mod{\varpi^e} \right\}$, as in the proof of Lemma \ref{ModifiedLemma6.4inBlue}. Then, by (\ref{eq:S0e>3}) -- (\ref{eq: S1 bound}), it follows that for $e=3k+l \geq 5$, where $l \in \{ 0,1,2\}$, we have
\begin{align*} 
S_{\varpi^e}(\v{0}) =& |\varpi|^{e+2n(k-1)}
				\begin{cases}
					\left(S_0(\varpi^3) - |\varpi|^{n-1}S_0(\varpi^2) \right),& \text{if } l = 0,\\
					\left(S_0(\varpi^4) - |\varpi|^{n-1}S_0(\varpi^3) \right),& \text{if } l = 1,\\
					\left(S_0(\varpi^5) - |\varpi|^{n-1}S_0(\varpi^4) \right),& \text{if } l = 2,
				\end{cases}\\
			=&  |\varpi|^{e+2nk}
				\begin{cases}
					1-|\varpi|^{-1}, &\text{if } e = 3k,\\
					S_1(\varpi) +1 - |\varpi|^{n-1}, &\text{if } e = 3k+1,\\
					|\varpi|^n - |\varpi|^{n-1}, &\text{if } e = 3k+2.
				\end{cases}
\end{align*}
Thus, $|\varpi|^{-en}S_{\varpi^e}(\v{0}) \ll |\varpi|^{-2n+5}$ for $e \geq 5$. Putting everything together, we obtain
\begin{align*}
 \mathfrak{S} 
 &=q^{-n+1} \prod_{\substack{ \varpi \text{ prime} \\ \varpi \neq t}}  \left(1 + O(|\varpi|^{-\frac{n}{2}+1}) \right),
\end{align*}
where the implied constant is independent of $q$.

Then, there exists a constant $c$, that is independent of $q$, such that
$$
\log q^{n-1}  \mathfrak{S}  = \sum_{\substack{ \varpi \text{ prime} \\ \varpi \neq t}}  \log \left(1 + \frac{c}{\varpi^{\frac{n}{2}-1}} \right) = \sum_{d\geq 1} \sum_{\substack{ |\varpi|=q^d \\ \varpi \neq t \\ \varpi \text{ prime}}} \sum_{m \geq 1}\frac{1}{m} \left(\frac{c}{\varpi^{\frac{n}{2}-1}} \right)^m  =  O(q^{2-\frac{n}{2}}),
$$
by the same argument as in (\ref{eq: number of primes}). Since $\exp(z) = 1 + O(|z|)$, we have $q^{n-1}  \mathfrak{S} = 1 + O(q^{2-\frac{n}{2}})$, which concludes the proof.
\end{proof}
Thus for $n \geq 10$, we have 
$$
M(d) = \mathfrak{S}  \mathfrak{J}|P|^{n-3}  + O(q^{11n/12+2/3}|P|^{3n/4-1}),
$$
where $\mathfrak{S}$ and $ \mathfrak{J}$ are given by Lemma \ref{Lemma: singular series} and (\ref{Lemma7.5inBlue}), respectively. Note that the error term is satisfactory for Theorem \ref{MainResult}. 
\section{Error term}
There is a satisfactory contribution to $N(d)$ from $|\theta| < \whQ ^{-5}$, since by (\ref{Eq4.1inBlue}), such terms contribute 
$$
 < \sum_{\substack{ r \text{ monic} \\ |r|\leq \whQ }} \whQ  \cdot \whQ ^{-5}|P|^n < \whQ ^{-3}|P|^n < |P|^{n-9/2} < |P|^{n-3}.
$$

Thus, we focus on the contribution from $|\theta| \geq \whQ ^{-5}$. As in \cite{BVBlue}, let $Y, \Theta \in \mbZ$ be such that
\begin{equation}\label{eq:BoundForYandTheta}
0 \leq Y \leq Q, \quad -5Q \leq \Theta < -(Y+Q).
\end{equation}
We will analyse the contribution to $N(d)$ coming from $\v{c}\neq \v{0}$ and $r, \theta$ such that $|r|=\whY $ and $|\theta|=\whTheta $. Denote this contribution by $E(d)=E(d; Y, \Theta)$. This section is similar to \cite[Section 7.4]{BVBlue} and \cite[Section 8]{BVBlue}, however we need to consider separately the cases when $t \mid r$ and $t\nmid r$. Thus, let
\begin{align*}
	B = 
  		\begin{cases}
			 0, &\text{if } t \mid r,\\
			 1, &\text{if } t \nmid r.
  		\end{cases}
\end{align*}
Moreover, note that Lemma \ref{Lemma4.4inBlue} imposes a constraint on $|\v{c}|$. More precisely,
$$
|\v{c}| \leq \whC =  q |\tilde{r}||P|^{-1}J(\Theta) = q^{B+1} |r||P|^{-1}J(\Theta),
$$
where  $\tilde{r} = rt^B$ and 
\begin{equation}\label{eq:Jtheta}
J(\Theta) = \max \left \{ 1, \whTheta |P|^3\right\}.
\end{equation}

Then $E(d)$ is given by
\begin{align*}
	&  \Sum_{\substack{\v{c} \in \xcal{O}^n \\  0 < |\v{c}| \leq  q^{B+1} \whY |P|^{-1}J(\Theta) }}  |P|^n  \Sum_{\substack{r \in \xcal{O}\\ |r| = \whY  \\ r \text{ monic}}} q^{-Bn}|r|^{-n} \int_{|\theta|=\whTheta } S_{r,t,\v{a}}(\v{c})I_{\tilde{r}}(\theta; \v{c})\diff\theta.
\end{align*}
By Lemma \ref{Lemma7.3inBlue}, $\left | I_{\tilde{r}}(\theta; \v{c}) \right | \ll L(\Theta)$, where 
\begin{equation}\label{eq:Ltheta}
L(\Theta) = \min \left\{q^{-n}, q^n \whTheta^{-\frac{n}{2}} |P|^{-\frac{3n}{2}} \right\}.
\end{equation}
Thus, $E(d)$ is 
\begin{equation}
 \ll |P|^n  \Sum_{\substack{\v{c} \in \xcal{O}^n \\ \v{c} \neq \v{0} \\|\v{c}| \leq  q^{B+1} \whY |P|^{-1}J(\Theta) }}  \quad \Sum_{\substack{r \in \xcal{O}\\ |r| = \whY  \\ r \text{ monic}}} q^{-Bn}|r|^{-n} \int_{|\theta|=\whTheta } \left| S_{r,t,\v{a}}(\v{c}) \right | L(\Theta) \diff\theta.  \label{lastEP}
\end{equation}
Moreover, since we must have $\v{c} \neq \v{0}$ and $\v{c} \in \xcal{O}^n$, we get the following bound 
\begin{equation}\label{LowerBoundOnY}
\whY \geq \frac{|P|}{J(\Theta)q^{B+1}}.
\end{equation}

Let $S$ be a set of finite primes to be decided upon in due course but which contains $t$. Any $r \in \xcal{O}$ can be written as $r=b_1'b_1''r_2$, where $b_1'$ is square free such that $\varpi \mid b_1' \Rightarrow \varpi \in S$ and $b_1''$ is square-free coprime to $S$. According to Lemma \ref{Lemma4.5inBlue} there exist $M_1, M_2 \in \{1,t\}$ such that $M_1 \mid b_1'$ and $M_2 \mid r_2$, together with $\v{b}_1,\v{b}_2 \in (\xcal{O}/t\xcal{O})^n$ such that
\begin{equation}\label{Lemma4.5inBlue}
|S_{r,t,\v{a}}(\v{c})|= |S_{b_1'',1,\v{0}}(\v{c})S_{b_1',M_1,\v{b}_1}(\v{c})S_{r_2,M_2,\v{b}_2}(\v{c})|.
\end{equation}
Clearly, $t\nmid b_1''$ for any $r$ and 
\begin{align*}
  		\begin{cases}
			M_1=t, M_2=1, &\text{if } t \mid \mid r,\\ 
    			M_1=1, M_2=t, &\text{if } t^2 \mid r,\\ 
			M_1=1, M_2=1, &\text{otherwise}. 
  		\end{cases}
\end{align*}
Moreover, by \cite[Lemma 2.2]{BVBlue} we have
$$
\int_{|\theta| = \whTheta} \diff\theta = \widehat{\Theta+1} - \whTheta \leq \widehat{\Theta+1}.
$$
Let $\xcal{O}^\sharp= \left \{b \in \xcal{O}: b \text{ is monic and square-free} \right\}$. There exist $\v{b}_1, \v{b}_2 \in (\xcal{O}/t\xcal{O})^n$ such that the bound for $E(d)$ in (\ref{lastEP}) becomes
\begin{align*}
	&\ll \hspace{-0.5cm} \Sum_{\substack{\v{c} \in \xcal{O}^n \\ \v{c} \neq \v{0} \\|\v{c}| \leq  q^{B+1} \whY |P|^{-1}J(\Theta) }} \hspace{-0.5cm} \frac{|P|^n \widehat{\Theta+1} L(\Theta)}{q^{B n} \whY ^{\frac{n-1}{2}}}  \Sum_{\substack{r_2 \in \xcal{O}\\ |r_2| \leq  \frac{\whY}{|b_1''b_1'|} }}  \hspace{0.2cm} \Sum_{\substack{b_1' \in \xcal{O}^\sharp\\ \varpi \mid b_1' \Rightarrow \varpi \in S}}  \frac{\left|S_{b_1',M_1,\v{b}_1}(\v{c})S_{r_2,M_2,\v{b}_2}(\v{c}) \right|}{|b_1'r_2|^{\frac{n+1}{2}}} \cdot \xcal{S},
 \end{align*}
	where
	$$
	\xcal{S} = \Sum_{\substack{b_1'' \in \xcal{O}^\sharp\\ (b_1'',S)=1 \\ |b_1'b_1''r_2| = \whY   }} \hspace{-0.3cm}  \frac{S_{b_1'',1,\v{0}}(\v{c})}{|b_1''|^{\frac{n+1}{2}}}.
$$

From now put $b=b_1''$, and $d=b_1'r_2$, for simplicity. Also, write $S_{b}(\v{c}) = S_{b,1,\v{0}}(\v{c})$. Moreover, take
$$
S = \begin{cases} 
 \left \{ \varpi : \varpi \mid  t F^*(\v{c}) \right \}, &\text{if } F^*(\v{c})\neq 0,\\
 \left \{ \varpi : \varpi \mid t \right \}, &\text{otherwise}.
\end{cases}
$$

\begin{lemma}\label{lemma: Sb}
We have
$$
 \xcal{S}  \ll
     		\begin{cases}
			\whY ^\varepsilon \frac{\whY }{|d|}, &\text{if }  F^*(\v{c}) \neq 0,\\
    			\whY ^\varepsilon \left(\frac{\whY }{|d|}\right)^{1+1/2}, &\text{otherwise},
  		\end{cases}
$$
where $F^*$ is the dual form of $F$. 
\end{lemma}
\begin{proof}
Recall that in our case $|\Delta_F|=1$. Furthermore, by \cite[Lemma 12]{Hea83} and \cite[Lemma 60]{Hoo94}, there exists a constant $ A(n) > 0$ depending only on $n$ such that for a prime $\varpi$ we have
\begin{equation}\label{Equation5.2inBlue}
 S_{\varpi}(\v{c}) \leq A(n) |\varpi|^{\frac{n+1}{2}}\left | \left (\varpi, F^*(\v{c}) \right) \right |^{1/2}. 
 \end{equation}
By Lemma \ref{Lemma4.5inBlue}, $S_{b}(\v{c})$ is a multiplicative function of $b$. Thus, by (\ref{Equation5.2inBlue}) and Lemma \ref{Lemma: tau_k} we have
$$
 \xcal{S} = \Sum_{\substack{b \in \xcal{O}^\sharp\\ (b,S)=1 \\ |bd| = \whY  }} \frac{|S_{b}(\v{c})|}{|b|^{\frac{n+1}{2}}} \ll \Sum_{\substack{b \in \xcal{O}^\sharp\\ (b,S)=1 \\ |bd| = \whY   }}  \left | \left (b, F^*(\v{c}) \right) \right |^{1/2} |b|^\varepsilon \ll \whY ^\varepsilon \Sum_{\substack{b \in \xcal{O}^\sharp\\ (b,S)=1 \\ |bd| = \whY  }}  \left | \left (b,  F^*(\v{c}) \right) \right |^{1/2}.
$$
The definition of $S$ and the constraint that $(b,S)=1$ imply that
\begin{align*}
\left (b,  F^*(\v{c}) \right) = 
  		\begin{cases}
			1, &\text{if }  F^*(\v{c}) \neq 0,\\
    			b, & \text{otherwise},
  		\end{cases}
\end{align*}
and this concludes the proof.
\end{proof}

Thus, we will consider separately the case when $F^*(\v{c}) \neq 0$ and the case when $F^*(\v{c}) = 0$. Denote the contributions to $E(d)$ coming from $\v{c}$ such that $F^*(\v{c}) \neq 0$, respectively $F^*(\v{c}) = 0$, by $E_1(d)$, respectively $E_2(d)$.

\subsection{Treatment of the generic term}\label{section:E1} 
Suppose $F^*(\v{c}) \neq 0$. Then, by the first part of Lemma \ref{lemma: Sb} we have
$$
E_1(d) \ll  \frac{|P|^n \widehat{\Theta+1} L(\Theta) \whY ^{\frac{3-n}{2} + \varepsilon} }{q^{B n} }  \Sum_{\substack{\v{c} \in \xcal{O}^n \\ \v{c} \neq \v{0} \\|c| \leq  q^{B+1} \whY |P|^{-1}J(\Theta) }} R_{1}(\v{c}),
$$
where
$$
R_{1}(\v{c}) = \Sum_{\substack{r_2 \in \xcal{O}\\ |r_2| \leq  \frac{\whY}{|b_1'|} }} \quad  \Sum_{\substack{b_1' \in \xcal{O}^\sharp\\ \varpi \mid b_1' \Rightarrow \varpi \in S}} \quad \frac{\left|S_{b_1',M_1,\v{b}_1}(\v{c})S_{r_2,M_2,\v{b}_2}(\v{c}) \right|}{|b_1'r_2|^{\frac{n+3}{2}}}.
$$
Then, (\ref{Equation5.2inBlue}) implies that 
\begin{align*}
R_{1}(\v{c})	
	&\ll \Sum_{\substack{b_1' \in \xcal{O}^\sharp\\ \varpi \mid b_1' \Rightarrow \varpi \in S}}  \frac{1}{|b_1'|^{1-\varepsilon}} \Sum_{\substack{r_2 \in \xcal{O}\\ |r_2| \leq  \whY }}  \quad \frac{\left|S_{r_2,M_2,\v{b}_2}(\v{c}) \right|}{|r_2|^{\frac{n+3}{2}}}
	\ll  \Sum_{\substack{r_2 \in \xcal{O}\\ |r_2| \leq  \whY }}  \quad \frac{\left|S_{r_2,M_2,\v{b}_2}(\v{c}) \right|}{|r_2|^{\frac{n+3}{2}}},
 \end{align*}
 since by Lemma \ref{Lemma: tau_k} we have 
 $$
 \Sum_{\substack{b_1' \in \xcal{O}^\sharp\\ \varpi \mid b_1' \Rightarrow \varpi \in S}}  \frac{1}{|b_1'|^{1-\varepsilon}}  = \prod_{\varpi \in S} \left(1 - \frac{1}{|\varpi|^{1-\varepsilon}} \right)^{-1} = \prod_{\varpi \in S} \Sum_{k=0}^{\infty} \frac{1}{|\varpi|^{k(1-\varepsilon)}} \leq \prod_{\varpi \in S} C_\varepsilon \ll  |P|^\varepsilon. 
 $$
 
Decompose $r_2$ as $b_2r_3$. Then, by the multiplicativity property in Lemma \ref{Lemma4.5inBlue}, we have $\left | S_{r_2,M_2,\v{b}_2}(\v{c}) \right| \leq \left |S_{b_2,M_2',\v{b}_2'}(\v{c}) S_{r_3,M_3,\v{b}_3}(\v{c})\right|$, for appropriate $M_2', M_3 \in \{ 1, t \}$ and $\v{b}_2', \v{b}_3 \in \left( \xcal{O}/t\xcal{O}\right)^n$. Thus,
\begin{align*}
 \Sum_{\substack{r_2 \in \xcal{O}\\ |b_1'r_2| \leq  \whY }} \frac{\left|S_{r_2,M_2,\v{b}_2}(\v{c})\right| }{|r_2|^{\frac{n+3}{2}}} & \leq  \Sum_{\substack{b_2r_3 \in \xcal{O}\\ |b_1'b_2r_3| \leq  \whY }} \frac{\left|S_{b_2,M_2',\v{b}_2'}(\v{c})S_{r_3,M_3,\v{b}_3}(\v{c})\right| }{|b_2r_3|^{\frac{n+3}{2}}}.
\end{align*}

Moreover, applying Lemma \ref{Lemma6.4inBlue} with $|M_3| \leq q$, 
\begin{align}\label{Sr3}
  \Sum_{\substack{\v{c} \in \xcal{O}^n \\ \v{c} \neq \v{0} \\|\v{c}| \leq  \whC  }}  \quad  \Sum_{\substack{r_3 \in \xcal{O}\\ |r_3| \leq  \frac{\whY }{| b_2|}  }} \frac{\left | S_{r_3, M_3, \v{b}_3}(\v{c}) \right |}{|r_3|^{\frac{n+3}{2}}} 
  &\ll q^n |P|^\varepsilon \left(\whY ^{n/3-1/6 } + \whC ^n \right).
\end{align}
\begin{lemma}\label{Lemma:Sb2M2'}
For $M_2'$ and $\v{b}_2'$ as above and $Y \in \mbZ$, there exists some $\varepsilon > 0$ such that
$$
\Sum_{\substack{b_2 \in \xcal{O}\\ |b_2| \leq  \whY }} \frac{\left|S_{b_2,M_2',\v{b}_2'}(\v{c})\right| }{|b_2|^{\frac{n+3}{2}}}   \ll q^{-2} |P|^\varepsilon,
$$
for any $\v{c} \in \xcal{O}^n$.
\end{lemma}
\begin{proof}
Suppose $M_2' =1$ so that 
$
S_{b_2,M_2',\v{b}_2'}(\v{c}) = S_{b_2,1,\v{0}}(\v{c}).
$
By \cite[(5.3)]{BVBlue}, together with the fact that in our case $|\Delta_F|=1$, there exists a constant $A(n)>0$ depending only on $n$ such that
\begin{equation}\label{Equation5.3inBlue}
S_{\varpi^2}(\v{c}) \leq A(n) |\varpi|^{n+1}\left| \left( \varpi, F^*(\v{c})\right) \right|.
\end{equation}
Then, by the multiplicativity property in Lemma \ref{Lemma4.5inBlue} and by (\ref{Equation5.3inBlue}), 
\begin{align*}
\Sum_{\substack{b_2 \in \xcal{O}\\ |b_2| \leq  \whY }} \frac{\left|S_{b_2,1,\v{0}}(\v{c})\right| }{|b_2|^{\frac{n+3}{2}}} &=  \Sum_{\substack{k_2 \in \xcal{O}\\ |k_2|^2 \leq  \whY }} \frac{\left|S_{k_2^2,1,\v{0}}(\v{c})\right| }{|k_2|^{n+3}} 
	  \ll \Sum_{\substack{k_2 \in \xcal{O}\\ |k_2| \leq  \whY ^{1/2}}} \frac{ \left|A(n)\right|^{\omega(k_2)} |k_2|^{n+1}\left| \left( k_2, F^*(\v{c})\right) \right|  }{|k_2|^{n+3}}.
	\end{align*}
It follows from Lemmas \ref{Lemma: tau_k} and \ref{TypicalSum1} that this can be bounded by	
	\begin{align*}
	\ll |P|^\varepsilon \Sum_{\substack{k_2 \in \xcal{O}\\ |k_2| \leq  \whY ^{1/2}}} \frac{ \left| \left( k_2, F^*(\v{c})\right) \right|  }{|k_2|^{2}} 
	 \ll |P|^\varepsilon \Sum_{\substack{k_2 \in \xcal{O}\\ |k_2| \leq  \whY ^{1/2}}} \frac{1  }{|k_2|} 
	 \ll  |P|^\varepsilon, 
	\end{align*}
as required. 

We now consider the case $M_2' = t$. First, we need to bound the sum
$$
S_{t^2,t,\v{b}_2'}(\v{c}) = \mathlarger{\sideset{}{^*}\sum}_{|a|<|t|^2} \Sum_{\substack{\v{y}\in \xcal{O}^n\\ |\v{y}| < |t|^2 \\ \v{y} \equiv \v{b}_2' \Mod t}} \psi \left( \frac{aF(\v{y})-\v{c}.\v{y}}{t^2} \right).
$$
Making a change of variables $\v{y}=\v{b}_2' + t\v{z}$, we get
\begin{align*}
S_{t^2,t,\v{b}_2'}(\v{c}) 
	&= \psi \left( \frac{-\v{c}.\v{b}_2'}{t^2} \right) \quad  \mathlarger{\sideset{}{^*}\sum}_{|a|<|t|^2} \psi \left( \frac{aF(\v{b}_2')}{t^2} \right) \Sum_{\substack{\v{z}\in \xcal{O}^n\\ |\v{z}| <|t| }} \psi \left( \frac{a\v{z}. \nabla F(\v{b}) - \v{c}.\v{z}}{t} \right).
\end{align*}
 But Lemma \ref{KubotaLemma7} implies that
 $$
  \Sum_{\substack{\v{z}\in \xcal{O}^n\\ |\v{z}| <|t|}} \psi \left( \frac{a\v{z}. \nabla F(\v{b}) - \v{c}.\v{z}}{t} \right) = 
  		\begin{cases}
			q^{n}, &\text{if } \left| a\nabla F(\v{b}) - \v{c}\right| <1,\\
    			0, &\text{otherwise},
  		\end{cases}
 $$
 and hence
\begin{equation}\label{Lemma:t^K}
|S_{t^2,t,\v{b}_2'}(\v{c})| \leq q^{n} \left| \{a \in \xcal{O} : |a| < |t|^2 :  \left| a\nabla F(\v{b}) - \v{c}\right| <1 \}\right| \leq q^{n}. 
\end{equation}

By the definition of $S_{b_2,M_2',\v{b}_2'}$ and the multiplicativity property in Lemma \ref{Lemma4.5inBlue}, we can write it as 
$ 
 S_{b_2,t,\v{b}_2'}(\v{c})= S_{t^2,t,\v{b}_2'}(\v{c})S_{(k_2/t)^2,1,\v{0}}(\v{c}).
$
The second sum is well understood and can be bounded using (\ref{Equation5.3inBlue}), giving
$$
\left|S_{(k_2/t)^2,1,\v{0}}(\v{c})\right|	
	 \ll |P|^\varepsilon |b_2|^{\frac{n+2}{2}} \frac{1}{q^{n+2}}.  
$$
Then, by (\ref{Lemma:t^K}) we have
$ 
 \left| S_{b_2,t,\v{b}_2'}(\v{c}) \right|  \ll q^{-2} |P|^\varepsilon |b_2|^{\frac{n+2}{2}}.  
 $ Hence,
$$
 \Sum_{\substack{b_2 \in \xcal{O}\\ |b_2| \leq  \whY }} \frac{\left|S_{b_2,t,\v{b}_2'}(\v{c})\right|}{|b_2|^{\frac{n+3}{2}}} \ll q^{-2} |P|^\varepsilon\Sum_{\substack{b_2 \in \xcal{O}\\ |b_2| \leq  \whY }} \frac{1}{|b_2|^{\frac{1}{2}}}
 	 = q^{-2} |P|^\varepsilon \Sum_{\substack{k_2 \in \xcal{O}\\ |k_2| \leq  \whY ^{1/2}}} \frac{1}{|k_2|}
	 \ll q^{-2} |P|^\varepsilon ,
$$
by Lemma \ref{TypicalSum1}.
\end{proof}
Thus, putting everything together,
\begin{align*}
\Sum_{\substack{\v{c} \in \xcal{O}^n \\ \v{c} \neq \v{0} \\|\v{c}| \leq  q^{B+1}  \whY J(\Theta)/|P| }} \hspace{-1cm} R_{1}(\v{c}) 
	& \ll q^{n-2} |P|^\varepsilon \left(\whY ^{n/3-1/6} + \left(q^{B+1} \whY |P|^{-1}J(\Theta)\right)^n \right),
\end{align*}
and hence,
\begin{align*}
E_1(d) 
 & \ll  q^{n-2} |P|^\varepsilon  \widehat{\Theta+1} L(\Theta) \left( \frac{|P|^n }{q^{Bn}\whY^{ \frac{n}{6}-\frac{4}{3}}} + q^n J(\Theta)^{n} \whY ^{\frac{n+3}{2}}\right).
 \end{align*}

Then, by (\ref{eq:Ltheta}) and (\ref{LowerBoundOnY}), the first term is
\begin{align*}
&\ll q^{\frac{7n}{6}-\frac{7}{3} - B(\frac{5n}{6} + \frac{4}{3})} \whTheta |P|^{\frac{5n}{6} + \frac{4}{3} + \varepsilon} J(\Theta)^{\frac{n}{6}-\frac{4}{3}} \min \left\{q^{-n}, q^n \whTheta^{-\frac{n}{2}} |P|^{-\frac{3n}{2}} \right\}.
\end{align*}
Noting that $B \in \{0, 1\}$ and $\min \{X, Z\} \leq X^uZ^v $ for any $u,v\geq 0$ such that $u+v=1$, by (\ref{eq:Jtheta}), we obtain
\begin{align*}
&\ll q^{\frac{7n}{6}-\frac{7}{3}-nu +nv}  \whTheta^{1-\frac{nv }{2}} |P|^{\frac{5n}{6} + \frac{4}{3}-\frac{3nv }{2} + \varepsilon} \max\left\{1, \whTheta |P|^3 \right\}^{\frac{n}{6}-\frac{4}{3}} .
\end{align*}
If $\whTheta |P|^3 \leq 1$, take $u = 1-\frac{2}{n}$ and $v = \frac{2}{n}$. Then, we obtain $\ll q^{\frac{n}{6}+\frac{5}{3}}  |P|^{\frac{5n}{6} - \frac{5}{3} + \varepsilon }$. Otherwise, if $\whTheta |P|^3 > 1$, take $u = \frac{2}{3}+\frac{2}{3n}$ and $v = \frac{1}{3}-\frac{2}{3n}$. Then, we get $\ll q^{\frac{5n}{6}-\frac{11}{3}}  |P|^{\frac{5n}{6} - \frac{5}{3} + \varepsilon}$.

Similarly, by (\ref{eq:Jtheta}) and (\ref{eq:Ltheta}), the second term is
\begin{align*}
&\ll q^{2n-1-nu+nv}  \whTheta^{1-\frac{nv}{2}} \max\left\{1, \whTheta |P|^3 \right\}^{n} |P|^{-\frac{3nv}{2}+\varepsilon}\whY^{\frac{n+3}{2}},
\end{align*}
for any $u,v\geq 0$ such that $u+v=1$. If $\whTheta |P|^3 \leq 1$, then by (\ref{eq:BoundForYandTheta}), we have
\begin{align*}
&\ll q^{2n-1-nu+nv}  \whTheta^{1-\frac{nv}{2}} |P|^{\frac{3n}{4}+\frac{9}{4}-\frac{3nv}{2}+\varepsilon}.
\end{align*}
Taking $u = 1 - \frac{2}{n}$ and $v = \frac{2}{n}$, we get $\ll q^{n+3} |P|^{\frac{3n}{4}-\frac{3}{4}+\varepsilon}$. Otherwise, if $\whTheta |P|^3 > 1$, we have
\begin{align*}
&\ll q^{2n-1-nu+nv}  \whTheta^{1+n-\frac{nv}{2}} |P|^{3n-\frac{3nv}{2}+\varepsilon}\whY^{\frac{n+3}{2}}.
\end{align*}
On noting that the exponent of $\whTheta$ is strictly positive for any $v \leq 1$, by (\ref{eq:BoundForYandTheta}), we have
\begin{align*}
&\ll q^{2n-1-nu+nv}|P|^{\frac{3n}{2}-\frac{3nv}{4}-\frac{3}{2}+\varepsilon}\whY^{-\frac{n}{2}+\frac{1}{2}+\frac{nv}{2}}.
\end{align*}
Taking $u = \frac{1}{n}$ and $v = 1- \frac{1}{n}$, we get $\ll q^{3n-3}|P|^{\frac{3n}{4}-\frac{3}{4}+\varepsilon}$.

Thus, for $n \geq 10$ we have $E_1(d) \ll q^{\frac{5n}{6}-\frac{11}{3}}  |P|^{\frac{5n}{6} - \frac{5}{3} + \varepsilon} + q^{3n-3}|P|^{\frac{3n}{4}-\frac{3}{4}+\varepsilon}$. Hence,
$$
E_1(d) \ll |P|^{\varepsilon} \left( q^{\frac{5(d+2)n}{6}-\frac{5d+16}{3}} + q^{\frac{3(d+5)n}{4}-\frac{3(d+5)}{4}} \right),
$$
which is satisfactory for Theorem \ref{MainResult}.

\subsection{Treatment of $E_2(d)$} Suppose $F^*(\v{c}) = 0$. Denote the contribution to $E(d)$ coming from $\v{c}$ such that $F^*(\v{c}) = 0$ by $E_2(d)$. Then, by the second part of Lemma \ref{lemma: Sb} we have
\begin{align*}
E_2(d) &\ll \frac{|P|^n \widehat{\Theta+1}  \whY ^{2 - \frac{n}{2} + \varepsilon}L(\Theta)}{q^{B n}} \hspace{-0.5cm}  \Sum_{\substack{\v{c} \in \xcal{O}^n \\ \v{c} \neq \v{0} \\|\v{c}| \ll  q^{B+1} \whY |P|^{-1}J(\Theta) }} \hspace{0.1cm} R_{2}(\v{c}),
 \end{align*}
 where
 $$
 R_{2}(\v{c}) =   \Sum_{\substack{r_2 \in \xcal{O}\\ |r_2| \leq  \frac{\whY}{|b_1'|} }}  \Sum_{\substack{b_1' \in \xcal{O}^\sharp\\ \varpi \mid b_1' \Rightarrow \varpi \in S}} \quad \frac{\left|S_{b_1',M_1,\v{b}_1}(\v{c})S_{r_2,M_2,\v{b}_2}(\v{c}) \right|}{|b_1'r_2|^{\frac{n}{2}+2}}.
 $$

As in Section \ref{section:E1}, decompose $r_2$ as $b_2r_3$.  Then, using Lemma \ref{ModifiedLemma6.4inBlue} and $|M_3|\leq q$, we get
  $$
  \Sum_{\substack{\v{c} \in \xcal{O}^n \\ \v{c} \neq \v{0} \\|\v{c}| \leq  \whC \\ F^*(\v{c}) = 0 }}  \quad  \Sum_{\substack{r_3 \in \xcal{O}\\ |r_3| =  \whR }} \frac{\left | S_{r_3, M_3, \v{b}_3}(\v{c}) \right |}{|r_3|^{\frac{n}{2}+2}} \ll q^{\frac{2n-5}{3}  +\varepsilon}   \whR^{\frac{n}{3} -\frac{1}{2} +\varepsilon } \left( 1+ \whC \right)^{n-\frac{3}{2}+\varepsilon}.
  $$
 Then, 
\begin{align}
  \Sum_{\substack{\v{c} \in \xcal{O}^n \\ \v{c} \neq \v{0} \\|\v{c}| \leq  \whC    \\ F^*(\v{c}) = 0 }}  \quad  \Sum_{\substack{r_3 \in \xcal{O}\\ |r_3| \leq  \frac{\whY }{|b_1' b_2|}  }} \frac{\left | S_{r_3, M_3, \v{b}_3}(\v{c}) \right |}{|r_3|^{\frac{n}{2}+2}} 
  &\ll q^{\frac{2n-5}{3} +\varepsilon}  \whY^{\frac{n}{3} -\frac{1}{2}  +\varepsilon } \left( 1+ \whC^{n-\frac{3}{2}+\varepsilon} \right)\label{BSr3}
\end{align}
Moreover, Lemma \ref{Lemma:Sb2M2'} implies that 
\begin{equation}\label{BSb21}
\Sum_{\substack{b_2 \in \xcal{O}\\ |b_2| \leq  \whY }} \frac{\left|S_{b_2,M_2',\v{b}_2'}(\v{c})\right| }{|b_2|^{\frac{n}{2}+2}} \ll q^{-2}|P|^\varepsilon.
\end{equation}
It follows from (\ref{Equation5.2inBlue}) that
\begin{align*}
R_2(\v{c}) 
	& \ll q^{-2} \Sum_{\substack{b_1' \in \xcal{O}^\sharp\\ \varpi \mid b_1' \Rightarrow \varpi \in S}} |b_1'|^{\varepsilon - 1} \Sum_{\substack{r_2 \in \xcal{O}\\ |b_1'r_2| \leq  \whY }} \frac{\left|S_{r_2,M_2,\v{b}_2}(\v{c})\right| }{|r_2|^{\frac{n}{2}+2}}.
\end{align*}
Then, by (\ref{BSr3}) and (\ref{BSb21}), we have
$$
\Sum_{\substack{\v{c} \in \xcal{O}^n \\ \v{c} \neq \v{0} \\|\v{c}| \leq  q^{B+1} \whY |P|^{-1}J(\Theta) }} \hspace{-1cm} R_2(\v{c})  \ll q^{\frac{2n}{3}-\frac{11}{3} }  \whY^{\frac{n}{3}-  \frac{1}{2}  + \varepsilon } \left( 1+ \left(q^{B+1} \whY |P|^{-1}J(\Theta)\right)^{n-\frac{3}{2}} \right).
$$

Thus, we can bound $E_2(d)$ by
\begin{align*}
	q^{\frac{2n}{3}-\frac{8}{3} } |P|^\varepsilon \whTheta \left( \frac{|P|^n L(\Theta) }{q^{Bn}\whY ^{\frac{n}{6} -\frac{3}{2}} } + q ^{n-\frac{3(B+1)}{2}}|P|^{\frac{3}{2}}J(\Theta)^{n-\frac{3}{2}}L(\Theta) \whY ^{\frac{n+1}{2}}\right).
\end{align*}

Then, by (\ref{eq:Ltheta}) and (\ref{LowerBoundOnY}), the first term is
\begin{align*}
&\ll q^{\frac{5n}{6}-\frac{7}{6} -B(\frac{5n}{6} -\frac{3}{2})} \whTheta  \min \left\{q^{-n}, q^n \whTheta^{-\frac{n}{2}} |P|^{-\frac{3n}{2}} \right\}  J(\Theta)^{ \frac{n}{6} -\frac{3}{2} }  |P|^{ \frac{5n}{6} +\frac{3}{2} +\varepsilon}.
\end{align*}
Since $B \in \{0, 1\}$ and $\min \{X, Z\} \leq X^uZ^v $ for any $u,v\geq 0$ such that $u+v=1$, by (\ref{eq:Jtheta}), we obtain
\begin{align*}
&\ll q^{\frac{5n}{6}-\frac{7}{6} -nu+nv} \whTheta^{1-\frac{nv}{2}} \max\left\{1, \whTheta |P|^3 \right\}^{ \frac{n}{6} -\frac{3}{2} }  |P|^{ \frac{5n}{6} +\frac{3}{2}-\frac{3nv}{2} +\varepsilon}.
\end{align*}
If $\whTheta |P|^3 \leq 1$, take $u = 1-\frac{2}{n}$ and $v = \frac{2}{n}$. Then, we obtain $\ll q^{-\frac{n}{6}-\frac{17}{6} } |P|^{ \frac{5n}{6} -\frac{3}{2} +\varepsilon}$. Otherwise, if $\whTheta |P|^3 > 1$, take $u=\frac{2}{3}+\frac{1}{n}$ and $v=\frac{1}{3}-\frac{1}{n}$. Then, we get $\ll q^{\frac{n}{2}-\frac{19}{6} } |P|^{ \frac{5n}{6} -\frac{3}{2} +\varepsilon}$.

Similarly, by (\ref{eq:Jtheta}), (\ref{eq:Ltheta}) and (\ref{LowerBoundOnY}), the second term is
\begin{align*}
&\ll q^{\frac{5n}{3}-\frac{25}{6} -nu +nv} \whTheta^{1-\frac{nv}{2}} |P|^{\frac{3}{2}-\frac{3nv}{2}+\varepsilon} \max\left\{1, \whTheta |P|^3 \right\}^{n-\frac{3}{2}} \whY ^{\frac{n+1}{2}},
\end{align*}
for any $u,v\geq 0$ such that $u+v=1$. If $\whTheta |P|^3 \leq 1$, then by (\ref{eq:BoundForYandTheta}), we have
\begin{align*}
&\ll q^{\frac{5n}{3}-\frac{25}{6} -nu +nv} \whTheta^{1-\frac{nv}{2}} |P|^{\frac{3n}{4}+\frac{9}{4}-\frac{3nv}{2}+\varepsilon}.
\end{align*}
Taking $u = 1 - \frac{2}{n}$ and $v = \frac{2}{n}$, we get $\ll q^{\frac{2n}{3}-\frac{1}{6} } |P|^{\frac{3n}{4}-\frac{3}{4}+\varepsilon}$. Otherwise, if $\whTheta |P|^3 > 1$, we have
\begin{align*}
\ll q^{\frac{5n}{3}-\frac{25}{6} -nu +nv} \whTheta^{n-\frac{1}{2}-\frac{nv}{2}} |P|^{3n - 3-\frac{3nv}{2}+\varepsilon} \whY ^{\frac{n+1}{2}}.
\end{align*}
Since the exponent of $\whTheta$ is strictly positive for any $v \leq 1$, by (\ref{eq:BoundForYandTheta}), we have
\begin{align*}
\ll q^{\frac{5n}{3}-\frac{25}{6} -nu +nv} \whY^{-\frac{n}{2}+1+\frac{nv}{2}} |P|^{\frac{3n}{2}-\frac{9}{4}-\frac{3nv}{4}+\varepsilon}.
\end{align*}
Taking $u = \frac{2}{n}$ and $v = 1- \frac{2}{n}$, we get $\ll q^{\frac{8n}{3}-\frac{49}{6}} |P|^{\frac{3n}{4}-\frac{3}{4}+\varepsilon}$.

Thus, for $n \geq 10$ we have $E_2(d) \ll q^{\frac{n}{2}-\frac{19}{6} } |P|^{ \frac{5n}{6} -\frac{3}{2} +\varepsilon} +q^{\frac{8n}{3}-\frac{49}{6}} |P|^{\frac{3n}{4}-\frac{3}{4}+\varepsilon}$. Hence,
$$
E_2(d) \ll |P|^{\varepsilon}\left(q^{\frac{(5d+8)n}{6}-\frac{3d}{2}-\frac{14}{3} } +q^{\frac{(3d+11)n}{4}-\frac{3d}{4}-\frac{107}{12}} \right),
$$
which is satisfactory for Theorem \ref{MainResult}.
\end{document}